\documentclass[ejs]{imsart}

\RequirePackage{amsthm,amsmath,amssymb}
\RequirePackage[numbers]{natbib}
\RequirePackage[colorlinks,citecolor=blue,urlcolor=blue]{hyperref}
\RequirePackage{graphicx}
\usepackage{amsfonts}
\usepackage{subfigure}

\startlocaldefs

\usepackage{mathbbol}
\usepackage{xcolor}

\DeclareSymbolFontAlphabet{\amsmathbb}{AMSb}%
\newcommand{\pr}{{\amsmathbb{P}}}
\newcommand{\E}{\amsmathbb{E}}
\newcommand*\dd{\mathop{}\!\mathrm{d}}
\newcommand\independent{\protect\mathpalette{\protect\independenT}{\perp}}
\def\independenT#1#2{\mathrel{\rlap{$#1#2$}\mkern2mu{#1#2}}}

\numberwithin{equation}{section}
\theoremstyle{plain}
\newtheorem{theorem}{Theorem}
\newtheorem{lemma}[theorem]{Lemma}

\newtheorem{definition}{Definition}
\newtheorem{remark}[theorem]{Remark}

\newtheorem{condition}{Condition}
\endlocaldefs

\begin{document}

\begin{frontmatter}
\title{Heterogeneous extremes in the presence of random covariates and censoring
\support{The work presented here is supported by the Carlsberg Foundation, grant CF23-1096.}
}
\runtitle{Heterogeneous extremes in the presence of random covariates and censoring}

\begin{aug}

\author{\fnms{Martin} \snm{Bladt}\ead[label=e1]{martinbladt@math.ku.dk}}
\address{Department of Mathematical Sciences, University of Copenhagen, Denmark\\
\printead{e1}}

\author{\fnms{Christoffer} \snm{Øhlenschlæger}
\ead[label=e2]{choh@math.ku.dk}}
\address{Department of Mathematical Sciences, University of Copenhagen, Denmark\\
\printead{e2}}

\end{aug}

\begin{abstract}
The task of analyzing extreme events with censoring effects is considered under a framework allowing for random covariate information. A wide class of estimators that can be cast as product--limit integrals is considered, for when the conditional distributions belong to the Fréchet max-domain of attraction. The main mathematical contribution is establishing uniform conditions on the families of the regularly varying tails for which the asymptotic behaviour of the resulting estimators is tractable. In particular, a decomposition of the integral estimators in terms of exchangeable sums is provided, which leads to a law of large numbers and several central limit theorems. Subsequently, the finite-sample behaviour of the estimators is explored through a simulation study, and through the analysis of two real-life datasets. In particular, the inclusion of covariates makes the model significantly versatile and, as a consequence, practically relevant.

\end{abstract}

\begin{keyword}[class=MSC]
\kwd[Primary ]{62G32}
\kwd[; secondary ]{62N02}
\kwd{62E20}
\end{keyword}

\begin{keyword}
\kwd{extreme values}
\kwd{survival analysis}
\kwd{random covariates}
\kwd{exchangeable decomposition}
\kwd{Potter bounds}
\end{keyword}
\tableofcontents
\end{frontmatter}

\section{Introduction}

There are several applications where data containing extreme values and censoring come into play. For instance, observations of survival times in medical studies are naturally right-censored and cured patients may live much longer than the average. Sometimes the latter effect can be so severe that the lifetime distributions of the cohort may adequately be modeled by a heavy-tailed distribution. Other examples arise when analysing catastrophic events, e.g. insurance claim sizes in certain lines of business may be very large and paid out throughout several years, incurring in incomplete information. A common feature of these type of studies is the presence of covariates, which are commonly the center of interest when drawing conclusions, for instance to quantify treatment effect differences on subpopulations, or determining risk classes within an insurance portfolio.

This paper is devoted to building the mathematical foundations for statistical inference of data exhibiting heavy-tails, censoring and covariates. To handle this problem, we employ an idea from \cite{EKM}, which analyses estimators derived as integrals with respect to the product-limit estimator (cf. \cite{kaplan1958nonparametric}) at normalized upper order statistics; however, we simultaneously incorporate the idea of \cite{stute-co}, where the inclusion of random covariates in handled elegantly in a censoring context. The extension is relevant on two fronts.
First, the incorporation of covariates greatly widens the practical relevance of models for heavy-tailed and right-censored data, enabling the use of a larger and more versatile group of estimators. Secondly, from a mathematical perspective there are several hurdles to overcome, which necessitates the formulation of appropriate and lax conditions on the distribution of the data and estimators which allows for tractable asymptotic behaviour. These conditions, in turn, lead to interesting features that are not present in the non-extreme counterparts in \cite{stute-co}. For instance, the covariate distribution changes support in the limit, concentrating on covariates giving rise to the most catastrophic events, and the original distribution is re-weighted according to the regularly-varying components of the conditional distributions. In applications, these features can be used asses not only how extreme events occur, but also which covariates are the drivers of such events.

The asymptotic properties of the product-limit estimator are by now well understood, with the main difficulty being dealing with maximal intervals. Thus, when considering integrals with respect to this estimator, the U-statistics approach of \cite{stute} is an elegant way of dealing with those issues, avoiding the use and complications that empirical process theory brings in this specific context. Here, a decomposition of the integral into sums (plus a remainder term) is crucial. In \cite{EKM} the Kaplan--Meier integral in without covariates is studied in detail, where a similar decomposition of the integral is established, but where the terms are no longer independent, but only exchangeable. In addition, central limit theorems are proven, in essence establishing conditions under which the dependence arising from the exchangeability asymptotically disappears. 

This paper extends the results from \cite{EKM} in the Fréchet max-domain of attraction and includes covariates. The main result is likewise a decomposition of extreme Kaplan--Meier integrals into exchangeable sums, and we establish that the remainder term again vanishes. However conceptually simple, the inclusion of covariates requires extending some fundamental results and bounds of univariate extreme value theory. Among others, we highlight the extension of the well-known Karamata representation and Potter's bounds to a uniform context.  Furthermore, we provide suitable uniform-type conditions under which we may establish that the dependence of the exchangeable sums disappears in the limit, allowing us to show consistency and three central limit theorems. As a by-product, we study the distribution of the covariates that lead to tail events, which is of importance for the limiting distribution in the CLT, but is also an interesting result on its own. Apart from technicalities, the general structure of the proofs follows \cite{EKM,stute,stute-co,segers}, with general facts and notation from extreme value theory following \cite{dehaan}. 

Other works have considered related estimators. In the context of censored extremes, the literature has grown in recent years. The extreme value index under random censoring was treated in \cite{beirlant2007estimation} and \cite{einmahl2008statistics} from basic principles, and through trimming in \cite{bladt2021trimmed}, while functional forms involving the product-limit estimator were proposed in \cite{worms2014new} and fully investigated in \cite{bww}; the extension to integral estimators was then treated in full generality in  \cite{EKM}. Some tangential considerations such as dependent censoring and truncation were treated in \cite{stupfler2019study} and \cite{gardes2015estimating}, respectively. Conditional tail index estimation was studied has been studied before, for instance through a parametric approach in \cite{index_regression} or through smoothing in \cite{gardes2014estimation}, and extended to right-censored settings in \cite{stupfler2016estimating}; however, these works concentrate onnly on the tail index and on estimating conditional distributions, and do not investigate or take into account other estimators, the evolution of the covariate distribution, nor the impact of said evolution on the proposed estimators. Classical extreme value theory was extended using sequential empirical processes in \cite{einmahl2016statistics} to non-identically distributed samples, without covariates, where each datapoint has a different contribution to the asymptotic behaviour. The proportionality function there plays a similar role as the re-weighting that happens to the asymptotic covariate distribution in our limit theorems. In \cite{de2021trends} the extreme value index is allowed to change continuously, allowing for the detection of trends, or deviations from constant extreme value indices, where the indexing conceptually plays the role of covariates. Finally, it is worth mentioning that \cite{EKM} also considered other max-domains of attraction, with a prime application being the analysis of the moment estimator of \cite{moment} adapted to censoring. Incorporating covariates in these lighter-tailed settings is a promising avenue of future research.

The paper is organized as follows. In Section \ref{sec:pre} we describe the setting and define necessary notation. The main results of the extreme Kaplan--Meier integrals re presented in Section \ref{sec:main}: the decomposition, consistency, and central limit theorems. Simulations studies based on a Burr model are presented in Section \ref{sec:application}, and we further showcase the utility of the estimators on real-life datasets from the medical and insurance sectors in Section \ref{sec:data}. 
Finally, Section \ref{sec:conc} concludes. Necessary groundwork for uniform-type results in univariate extremes are provided in Appendix \ref{sec:potter} and  \ref{sec:segers}. The proofs of the main results are delegated to Appendix \ref{sec:proofs-main} and \ref{sec:proof-X}. Appendix \ref{sec:Burr} shows that the Burr model used in Section \ref{sec:application} satisfies the conditions of Section \ref{sec:main}.

\section{Preliminaries}
\label{sec:pre}
\subsection{Notation and Setting}

Let $(Y,X,C)$ be a random vector on the probability space $(\Omega,\pr,\mathcal{F})$. Here, $Y\in \amsmathbb{R}$ is a variable of interest and $X\in\mathcal{X}\subset \amsmathbb{R}^m$ is a vector of covariates. We assume right-censoring, so we only observe
\begin{align*}
    X,\quad Z=\min\{Y,C\}\quad \text{and} \quad \delta=1\{Y\leq C\}
\end{align*}
where $C$ is understood as a censoring mechanism. Let $(Z_1,\delta_1,X_1),\ldots, (Z_n,\delta_n,X_n)$ be $n$ independent samples of $(Z,\delta,X)$. The dependence between the different random components requires some structure if one is to recover the behaviour of $(Y,X)$ from the observed sample. For instance, \cite{stute-co} studied the full distribution of $(Y,X)$, where the assumption
\begin{align*}
\pr(Y\le C| X,Y)=\pr(Y\le C| X)\quad \text{and}\quad Y\independent C
\end{align*}
is crucial to identify the model. Here, we instead adopt the assumption
\begin{align}\label{dependence_assumption}
\pr(Y\le C| X,Y)=\pr(Y\le C| X) \quad \text{and}\quad (Y,X)\independent C
\end{align}
which also allows for arbitrary dependence between $X$ and $Y$, though the censoring mechanism $C$ cannot depend on $X$. That being said, we are interested in the tail behaviour of $Y$, so that independence relation \eqref{dependence_assumption} can be relaxed to the equivalent tail-independence relation. One possible consequence is that $C$ is allowed to depend on the covariates $X$ in the body of the distribution -- though still not in the tail. An avenue to include the dependence, and conditions on the rates of asymptotic independence, is through a copula construction as in \cite{stupfler2019study}. To avoid notational overhead, and to streamline proofs, we stick with \eqref{dependence_assumption}.

The next step is to construct an estimator which can overcome the censoring effects while still concentrating on the tail. Since we consider random covariates, it is natural to consider an estimator inspired by \cite{EKM,stute-co} and given by
\begin{align}\label{def:EKM}
    \amsmathbb{F}_{k,n}(x,y)=\sum_{i=1}^kW_{ik}\,1\{X_{[n-i+1,n]}\leq x, Z_{n-i+1,n}/Z_{n-k,n}\leq y\}
\end{align}
where the product weights are given by
\begin{align*}
    W_{ik}=\frac{\delta_{[n-i+1:n]}}{i}\prod_{j=i+1}^{k}\Big[\frac{j-1}{j}\Big]^{\delta_{[n-j+1:n]}}
\end{align*}
and $Z_{1:n}\leq Z_{2:n}\leq \ldots \leq Z_{n:n}$ are the order statistics of the $(Z_i)$ and $\delta_{[i:n]}$ and $X_{[i:n]}$ are the associated concomitants, that is, the $\delta$ and $X$ values associated with $Z_{i:n}$. The value $k=k_n$ is an intermediate sequence which satisfies
\begin{align}\label{speedk}
k\to\infty  \quad \text{and}\quad\frac{k}{n}\to 0, \quad \text{as}\quad n\to\infty.
\end{align}
This assumption is standard when focusing on the tail of a distribution, and states that the number of top order statistics that we use to construct our estimator becomes large, though their relative proportion with respect to the sample vanishes. The latter guarantees that we concentrate on the tail behaviour.

As noted in \cite{stute}, it turns out that recovering the distributional properties of the target variables from censored data requires expressing product-limit estimators in terms of simpler components. For instance, when covariates are present, \cite{stute-co} used the following functions to express the integral of a generic function $\varphi$ with respect to the product limit estimator:
\begin{align*}
     \gamma_0(y)&=\exp\Big\{\int_0^{y-}\frac{H^0(\dd z)}{1-H(z)}\Big\} \\
     \gamma_1(y)&=\frac{1}{1-H(y)}\int 1_{y<w}\varphi(x,w)\gamma_0(w)H^{11}(\dd x,\dd w) \\
     \gamma_2(y)&=\int \int \frac{1_{v<y,v,w}\varphi(x,w)\gamma_0(w)}{(1-H(v))^2}H^0(\dd v)H^{11}(\dd x,\dd w)
\end{align*}
where 
 \begin{align*}
     &H(z)=\pr(Z\leq z),\quad H^{11}(x,y)=\pr(X\leq x, Z\leq y, \delta=1)\\
     & H^0(z)=\pr(Z\leq z, \delta=0).
 \end{align*}
A modification of these functions to the tail setting is introduced below, which turns out to be the precise adaptation required to study the asymptotic behaviour of \eqref{def:EKM}. It must also be noted that when studying integrals of functions $\varphi$ with respect to \eqref{def:EKM}, the limit result depends on how $k$ is chosen. In the simulations of Section \ref{sec:application} we describe the situation in further detail. 

\subsection{Extreme value theory}
We consider the case when both $Y$ and $C$ are in the Fr{é}chet max-domain of attraction, that is, when their tails $1-F_Y$ and $1-G$ are roughly polynomially decaying. In \cite{EKM}, all max-domains of attraction were considered when covariates are not present, though the treatment of the Weibull and Gumbel max-domains is particularly complicated and assumption-heavy. Since the introduction of covariates brings along its own set of obstacles and assumptions, we restrict ourselves to studying arguably the most important case: regularly varying tails.

More precisely, let 
\begin{align*}
\gamma_G>0\quad \text{and}\quad \gamma_F(x)\in [\gamma_F^L,\gamma_F^U]
\end{align*}
where $0<\gamma_F^L<\gamma_F^U<\infty$. We assume that for $x\in\mathcal{X}$
\begin{align*}
F_{Y|X=x}(y)&=\pr(Y\le y| X=x)=L_{Y|X=x}(y)\,y^{-1/\gamma_F(x)}\\
G(y)&=\pr(C\le y)=L_C(y)\,y^{-1/\gamma_G}
\end{align*} 
where $L_{Y|X=x}$ and $L_C$ are slowly varying functions at infinity\footnote{A function $\ell$ is slowly varying at infinity if for any $t>1$ we have $\lim_{x\to\infty} \ell(tx)/\ell(x)=1$.}. This allows for heterogeneity in the extreme values of $Y$. We again highlight that the heterogeneity of $C$ can be part of the model under tail independence in \eqref{dependence_assumption} in place of usual independence, though even in that setting the tail index $\gamma_G$ of the censoring mechanism would still not depend on $x$. The independence of $Y|X$ and $C$ implies that $Z$ given $X$ is also regularly varying with tail index $\gamma_H(x)$ where
\begin{align*}
    1/\gamma_H(x)=1/\gamma_F(x)+1/\gamma_G.
\end{align*}

Concerning the covariates, and since we have left the distribution of the vector $(Y, X)$ unspecified, it is important to study how their distribution evolves as we sample from large values of (the censored version of) $Y$. For this purpose, define
\begin{align*}
    F_X^t(x)=\pr(X\le x|Z>t).
\end{align*}

 Further, we define $F^t_{Y|X=x}$ and $G^t_Y$ as:
\begin{align*}
    F^t_{Y|X=x}(y)&=\pr(Y\leq yt|Y>t,X=x)=\frac{F_{Y|X=x}(yt)-F_{Y|X=x}(t)}{1-F_{Y|X=x}(t)}, \quad y\geq 1, \\
    G^t(y)&=\pr(C\leq yt|C>t)=\frac{G(yt)-G(t)}{1-G(t)}, \quad y\geq 1.
\end{align*}
As $F_{Y|X=x}$ and $G$ belong to the Fréchet max-domain of attraction, we have that $F_{Y|X=x}^t$ and $G^t$  converge to the Pareto distribution with respective tail indices $\gamma_F(x)$ and $\gamma_G$. Generally, we call $F_R^t(x)$ the \textit{tail counterpart} of $F_R$, and $U_{F_R}(x)=\inf\{y:F_R(y)\ge 1-1/x\}$ the \textit{tail quantile function}. 

We are now in a position to state the required adaptation of the $\gamma_1,\gamma_2,\gamma_3$ functions from the previous subsection. Namely, let $\gamma_1^t, \gamma_1^t$ and $\gamma_2^t$ denote the respective functions where $H$, $H^0$ and $H^{11}$ are simply replaced by their tail counterparts $H^t$, $H^{0,t}$ and $H^{11,t}$.

Any regularly varying function can be expressed in terms of a so-called Karamata representation, and thus the same is true for $F_{Y|X=x}$ for a single $x$. The same is true for the tail quantile functions $U_{F_{Y|X=x}}$. We now define a uniform Karamata representation for a family of regularly varying functions, which reduces to the non-uniform version when the family consists of one member.

\begin{definition}[Uniform Karamata representation]
    Let $h(x,y):\mathcal{X}\times \amsmathbb{R}\to\amsmathbb{R}$ be regularly varying at infinity in $y$ for any fixed $x$. We say that $h$ has a uniform Karamata representation if  
    \begin{align*}
        h(x,y)=c(x,y)\exp\left( \int_1^x \delta(x,u)/u du \right), \quad x\geq 1
    \end{align*}
where $c$ and $\delta$ are measurable functions satisfying, as $y\rightarrow \infty$, $c(x,y)\rightarrow c(x)$ uniformly and $\delta(x,y)\rightarrow \delta(x)$ uniformly.
If this representation holds, we say that $h(x,y)$ can be uniformly decomposed. 
    \label{def:decompose}
\end{definition}

\section{Main results}
\label{sec:main}

This section provides the main theoretical results concerning integrals of functions $\varphi$ with respect to the estimator \ref{def:EKM}. The main result is the decomposition of Theorem \ref{thm:decomposition}, which then allows to establish consistency and normality under different conditions. Naturally, the stronger the conditions, the more tractable the weak convergence results become. It must be noted, however, that the inclusion of covariates requires the strengthening of many extreme value theory results, uniformly across covariates. Notable examples are the uniform Potter bounds of Appendix \ref{sec:potter} and uniform Segers bounds of Appendix \ref{sec:segers}.

Concretely we define the Extreme Kaplan--Meier integral as
\begin{align}
    S_{k,n}(\varphi)=\sum_{i=1}^kW_{ik} \varphi(X_{[n-i+1,n]},Z_{n-i,n})=\int \varphi \dd \amsmathbb{F}_{k,n}
\end{align}
where $\varphi: \amsmathbb{R}^{m+1}\rightarrow \amsmathbb{R}$ is a measurable function. To ensure that integrating $\varphi$ is possible, we introduce the following definition and condition.
\begin{definition}[Uniform envelope]
    A continuous and monotone function $\hat{\varphi}$ on $[1,\infty]$ such that $|\varphi(x,w)|\leq \hat{\varphi}(w)$ is called an uniform envelope of $\varphi$.
\end{definition}

\begin{condition}[Finite moment]
    There exists an envelope $\hat{\varphi}$ of $\varphi$ such that
    \begin{align*}
        \int_1^{\infty} \hat{\varphi}(w)^{2+\varepsilon}w^{\alpha(\epsilon)}\dd w<\infty
    \end{align*}
    where $\alpha(\varepsilon)=1/\gamma_G-1/\gamma_F^U-1+\varepsilon$.
    \label{cond-full}
\end{condition}
Our results can be established under slightly laxer conditions than Condition \ref{cond-full}, though the above is rather easy to verify. Thus, the use of envelopes could in principle be foregone, but at the cost of clarity and applicability.

For technical reasons, we assume the underlying probability space is rich enough to support the following random vectors
\begin{align*}
(X^*, Y^*/t, C^*/t)\sim\mathcal{L}\left( X, Y/t, C/t |Z>t\right).
\end{align*}
Although we have suppressed the dependence in $t$, note that the marginal distributions of the above quantities are related to the notion of tail counterpart as follows: $X^*\sim F_{X}^t$, $Y^*/t\sim F^t_Y$ and $C^*/t \sim G^t$. Then let 
\begin{align*}
\big((X_{[1,k]},Y^*_{1,k}/t, C^*_{1,k}/t), \ldots ,(X_{[k,k]},Y^*_{k,k}/t, C^*_{k,k}/t)\big)\end{align*}
be the corresponding order statistics and concomitants from iid sample of size $k$. Then 
\begin{align*}
\Big(\big(X_{[n-k+1,n]},\frac{Z_{n-k+1,n}}{Z_{n-k,n}}, \delta_{[n-k+1,n]}\big), \ldots ,\big(X_{[n,n]},\frac{Z_{n,n}}{Z_{n-k,n}},\delta_{[n,n]}\big)\Big) \mid Z_{n-k,n}=t\\
\sim ((X_{[1,k]}^*, V_{1,k}^*,\delta_{[1,k]}^*), \ldots, (X_{[k,k]}^*, V_{k,k}^*,\delta_{[k,k]}^*))
\end{align*}
where $V_i^*=\min\{Y_i^*/t, C_i^*/t\}$ and $\delta_i^*=1(Y_i^*\leq C_i^*)$. \\

\subsection{Decomposition and consistency}
The main result of this section is the following decomposition.

\begin{theorem}[Main decomposition]\label{thm:decomposition}
    Let $\varphi$  satisfy Condition \ref{cond-full}. Further, assume that $U_{F_{Y|X=x}}$ can be uniformly decomposed. Then
    \begin{align*}
        &r_{k,n}=\\
        &S_{k,n}(\varphi)-\Big\{ \frac{1}{k}\sum_{i=1}^k \varphi(X^*_{i},V_i^*)\gamma_0^t(V_i^*)\delta_i^*+\frac{1}{k}\sum_{i=1}^k\gamma_1^t(V_i^*)(1-\delta_i^*)-\frac{1}{k}\sum_{i=1}^k\gamma_2^t(V_i^*)\Big\}
    \end{align*}
     is such that for every $\epsilon>0$ 
     \begin{align*}
         \lim_{n\rightarrow \infty}\limsup_{t\rightarrow \infty} \pr(|k^{1/2}r_{k,n}|>\epsilon|Z_{n-k,n}=t)=0.
     \end{align*}
    \label{3.2}
\end{theorem}
By Theorem \ref{3.2} it is possible to decompose $S_{k,n}(\varphi)$ into simple exchangeable sums modulo a uniformly negligible term $r_{k,n}$. For large samples, the exchangeability of the sums has a law of large numbers behaviour, as is seen from the following consistency result. Let $F$ be the multivariate distribution function of the vector $(Y,X)$.

\begin{theorem}[Weak consistency]
    Assume the same conditions as in Theorem \ref{3.2}. Then
    \begin{align*}
        S_{k,n}\stackrel{\pr}{\to} S_{\circ}(\varphi)
    \end{align*}
where $S_{\circ}=\int_1^{\infty}\varphi \dd F^{\circ}$ and $F^{\circ}$ is the limit distribution of $F^t$ as $t\rightarrow \infty$.
    \label{thm:weak-con}
\end{theorem}

The limit distribution of $F^t$ can be evaluated through writing it as $F^t(x,y)=F^t_{Y|X=x}(y)F^t_X(x)$. The first term in the product is regularly varying, which has a Pareto limit as $t\to\infty$. The second term in the product is more subtle and is treated in the following lemma.
\begin{lemma}[Limit of the covariate distribution]
    Assume that $(x,y)\mapsto F_{Y|X=x}(y)$ can be uniformly decomposed. In addition, assume that the $c$-function from its Karamata's representation is eventually bounded. Let $B=\{x\in\mathcal{X}:\gamma_F(x)\ge \gamma_F(y),\forall y\in \mathcal{X}\}$. Then for $B_1\subset \mathcal{X}\setminus B$ with $\pr(X\in B_1)>0$, we have
    \begin{align*}
         \pr(X\in B_1 |Y>t)\rightarrow 0,
    \end{align*}
    while $B_1\subset B$ with $\pr(X\in B_1)>0$ we have
    \begin{align*}
        \lim_{t\to \infty}\pr(X\in B_1 |Y>t)=\lim_{t\to \infty}  \frac{\int_{B_1}L_{Y|X=x}(t)F_X(\dd x)}{\int_B L_{Y|X=x}(t)F_X(\dd x)},
    \end{align*}
    if it exists.
    \label{FX-dist}
\end{lemma}

In other words, the covariate limit distribution is concentrated on the set of those $x$ which yield the largest attainable tail index, and on that set, the original $F_X$ distribution is re-weighted according to the slowly varying functions.

\subsection{Central limit theorems}
We turn our attention to the central limit theorem which is attainable directly from Theorem \ref{thm:decomposition}, and where the centering sequence depends on $Z_{n-k,n}$ and hence is random. An important component of the limit is its variance. For this purpose, we require the specification of
\begin{align*}
    W^{\circ}=\varphi(X^{\circ},V^{\circ})\gamma_0^{\circ}(V^{\circ})\delta^{\circ}+\gamma_1^{\circ}(V^{\circ})(1-\delta^{\circ})-\gamma_2^{\circ}(V^{\circ})
\end{align*}
where ${\circ}$ denotes the limit variable or distribution as $t\rightarrow \infty$. We then have the following theorem.

\begin{theorem}[Central limit theorem]
    Under the conditions of Theorem \ref{3.2}
    \begin{align*}
        \sqrt{k}\int \varphi \dd (\amsmathbb{F}_{k,n}-F^{Z_{n-k,n}})\stackrel{d}{\to}\mathcal{N}(0,\sigma_{\varphi}^2)
    \end{align*}
    where $\sigma^2_{\varphi}=\mbox{Var}(W^{\circ})$.
    \label{thm:normal1}
\end{theorem}
The next step now consists of showing that the random centering sequence $\int \varphi \dd F^{Z_{n-k,n}}$ can be replaced with a non-random centering sequence $\int \varphi \dd F^{U_H(n/k)}$, though we require some additional conditions to ensure that $\varphi$ behaves well enough. 

\begin{condition}[Lipchitz and derivative behaviour of $\varphi$]
    Let $\hat{\varphi}'$ be the envelope of $\frac{\partial}{\partial y}\varphi(x,y)$.  For some $\epsilon>0$
    \begin{align*}
        \int_1^{\infty} \hat{\varphi}'(w) w^{-1/\gamma_F^U+\epsilon}\dd w<\infty.
    \end{align*}
    Additionally, for all $x\in\mathcal{X}$
     \begin{align*}
        |\varphi(x,y)-\varphi(x,z)|\leq M |\hat{\varphi}(y)-\hat{\varphi}(z)| \text{ for a constant }M.
    \end{align*}
    \label{cond:assyp}
\end{condition}

\begin{definition}[Uniform normalized tail quantile functions] We say that $U_{F_{Y|X=x}}(t)$ is uniform normalized if there exists a $T$ such that
    \begin{align*}
        U_{F_{Y|X=x}}(t)=U_{F_{Y|X=x}}(T)\exp\left(\int_T^t\eta_x(u) \dd u/u\right) \quad \forall x\in \mathcal{X}
    \end{align*}
    where $\eta_x(u)$ is a real measurable function uniformly converging to $\eta_x=\gamma_F(x)$.
    \label{def:uni-U}
\end{definition}

Finally, since we have introduced some new notions to control the behaviour of $\varphi$ and $U_{F_{Y|X=x}}$, we require a more precise speed of convergence than \eqref{speedk}.
\begin{condition}[Speed on the $k$ sequence]
Let $\eta_x(\cdot)$ be differentiable. The sequence $k$ satisfies that
    \begin{align*}
        \sup_{x^0,x^1\in\mathcal{X}}\log(n/k)\int_{U^{\leftarrow}_{F_{Y|X=x^0}}(U_H(n/k))}^{\infty} |\eta'_{x^1}(l)|\dd l \rightarrow 0\\
        \log\left(\frac{L_{x^1}(U_H(n/k))}{L_{x^0}(U_H(n/k))}\right)=o\left(\log(n/k)\left|\frac{1}{\gamma_F(x^0)}-\frac{1}{\gamma_F(x^1)}\right|\right) 
    \end{align*}
    uniformly in $x^0,x^1\in \mathcal{X}$
\label{cond:slowly}
\end{condition}

\begin{theorem}[CLT under Lipschitz condition]
    Assume the same conditions as in Theorem \ref{3.2}. Further assume Condition \ref{cond:assyp}, and that $U_{F_{Y|X=x}}$ is uniform normalized. Assume that $c_{L_{Y|X=x}}$ is bounded, and there exists $\epsilon>0$ such that eventually $-u^{-\epsilon}<\delta_{L_{Y|X}}(x,u)<u^{-\epsilon}$, where $c_{L_{Y|X}}(x)$ and $\delta_{L_{Y|X}}$ is respectively the $c$ and $\delta$-function from the Karamata representation of $L_{Y|X}$. Finally let $k$ satisfy $k=o(\log^b(n^a))$ for some $b<2,\,a>0$, and Condition \ref{cond:slowly}.

 Then 
    \begin{align}
        \sqrt{k}\int \varphi \dd (F^{Z_{n-k,n}}-F^{U_H(n/k)}) \stackrel{\pr}{\to} 0.\label{THM35}
    \end{align}

As a direct consequence
    \begin{align*}
        \sqrt{k} \int\varphi \dd (\amsmathbb{F}_{k,n}-F^{U_H(n/k)})\stackrel{d}{\rightarrow} \mathcal{N}(0,\sigma^2_{\varphi}).
    \end{align*}
    \label{THM:nlast}
\end{theorem}

We provide a final refinement of the central limit theorem where there no longer is a centering sequence. As is common in extreme value statistics, knowing the second-order behaviour of the conditional distributions is required.

\begin{condition}[Second-order condition]
    For $x\in B$
    \begin{align*}
        \lim_{t\rightarrow \infty} \frac{U_{F_{Y|X=x}}(ty)/U_{F_{Y|X=x}}(t)-y^{\gamma_F(x)}}{a_{X=x}(t)}=y^{\gamma_F(x)}h_{\rho(x)}(y), \quad y\geq 1
    \end{align*}
uniformly, with $a_{X=x}(t)$ positive or negative, regularly varying with index $\rho(x)\leq 0$ and satisfying $a_{X=x}(t)\rightarrow 0$ as $t\rightarrow \infty$ and
    \begin{align*}
        h_{\rho(x)}(y)=\begin{cases}
            \log (y) & \text{if $\rho(x)=0$} \\
            (y^{\rho(x)}-1)/\rho(x) & \text{otherwise}
        \end{cases}
    \end{align*}
    We furthermore require that $\max_x \rho(x)<0$ when $F_X$ is absolutely continuous.
    \label{cond:second-order}
\end{condition}

The last theorem can now be stated.

\begin{theorem}[CLT with second-order condition]
    Let $k=o(\log^b(n^a))$ for some $b<2,\,a>0$. Assume the same conditions as in Theorem \ref{THM:nlast}. Further assume that $a_{X=x}$ can be uniformly decomposed with $\min_x c(x)>0$, and that $U_{F_{Y|X=x}}$  satisfies the second order Condition \ref{cond:second-order}. Let $\sqrt{k}a_{X=x}(c_n(x)) \rightarrow \lambda(x)\geq 0$ with $c_n(x)=(1-F_{Y|X=x}(U_H(n/k)))^{-1} $ uniformly, and if $\lambda(x)>0$ assume additionally that $\frac{\partial}{\partial y}\varphi(x,y)$ is a.e. continuous in $y$ and that  $\sqrt{k}a_{X=x}(c_n(x))$  is eventually uniform bounded. Then 
    \begin{align*}
    \sqrt{k} \int\varphi \dd (F^{Z_{n-k,n}}-F^{\circ})\stackrel{\pr}{\to} \int \lambda(x)C(\gamma_F(x),\rho(x)) F^{\circ}_X(\dd x)
\end{align*}
    where 
    \begin{align*}
        C(\gamma_F(x),\rho(x))=\int_1^{\infty}x \big(\frac{\partial}{\partial y}\varphi\big)(x,y)h_{\rho(x)}(y^{1/\gamma_F(x)})F^{\circ}_{Y|X=x}(\dd y).
    \end{align*}
    \label{THM:last}
As a direct consequence
    \begin{align*}
        \sqrt{k} \int\varphi \dd (\amsmathbb{F}_{k,n}-F^{\circ})\stackrel{d}{\rightarrow} \mathcal{N}\Big(\int \lambda(x)C(\gamma_F(x),\rho(x)) F^{\circ}_X(\dd x),\sigma^2_{\varphi}\Big).
    \end{align*}  
\end{theorem}
Note that once again the result depends on the limit distribution of $X$. The distribution can be calculated according to Lemma \ref{FX-dist}.

\section{Applications}
\label{sec:application}

This section provides two numeric studies on simulated data. As the data generating mechanism for $Y|X$, we use the Burr distribution with density
\begin{align*}
    f(y|x)=\frac{\kappa(x)\tau(x) y^{\tau(x)-1}}{(1+y^{\tau(x)})^{\kappa(x)+1}},
\end{align*}
which is a popular generalization of the Pareto distribution, having significant bias at lower quantiles. The distribution of $X$ is uniform on $[1,5]$. Concerning the censoring mechanism, we let $C=C_11\{C_1<100\}+C_21\{C_1\geq 100\}$, where $C_1,C_2$ are independent Pareto distributions with tail indices given by
\begin{align*}
    &\gamma_{C_1}(x)=0.75+3 \phi((x-2)/0.1)+3 \phi((x-4)/0.1) \\
    &\gamma_{C_2}=14
\end{align*}
where $\phi$ is the density for a standard normal variable. 
Similarly, the tail index for $Y$ is
\begin{align*}
\gamma_F(x)=0.5+2\phi((x-2)/0.1)+2\phi((x-4)/0.1).
\end{align*}
The parameters we have chosen for the Burr distribution of the conditional response variable are
\begin{align*}
    \kappa(x)=1, \quad \tau_F(x)=1/\gamma_F(x).
\end{align*}
Through this construction we guarantee that $C$ is independent of $X$ in the tail. Moreover, our construction allows for $C$ to depend on $X$ at lower quantiles. However, if we were to only use $C_2$, then there would only be a minimal of non-large observations, which would be censored.
The functions $\gamma_F(x)$ and $\gamma_{C_1}(x)$ are plotted in Figure \ref{fig:gamma}. By the choice of $\gamma_F$ we have tail $\gamma_F(2)=\gamma_F(4)=\gamma_F^U$. In other words, the asymptotic covariate distributions is degenerate on two points, with $\pr(X^{\circ}=2)=\pr(X^{\circ}=4)=0.5$. The model setup and the $\varphi$ functions considered below also satisfy all the  assumptions for the asymptotics of this paper, as is verified in Appendix \ref{sec:Burr}.
\begin{figure}[htbp]
    \centering
        \includegraphics[width=0.45\textwidth]{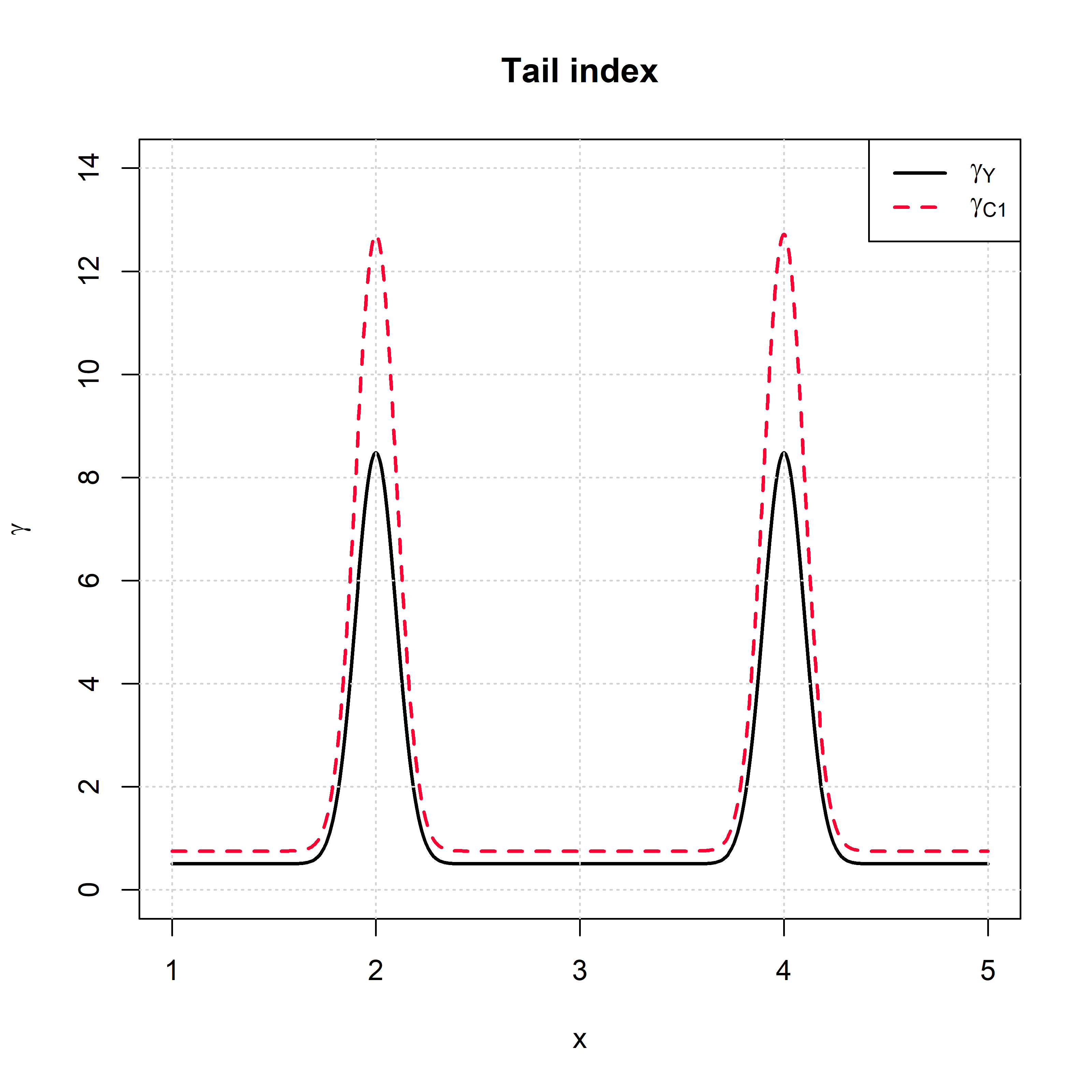}
    \caption{Tail index functions for the conditional response and censoring mechanism employed in the simulation study.} 
    \label{fig:gamma}
\end{figure}

Finally, we set $Z=\min\{Y,C\}$ and $\delta=1\{Y\leq C\}$, and generate an i.i.d. sample of $n=10^5$ observations for the simulation studies.

\subsection{Finding homogeneous and catastrophic extremes}
\label{sec:Finding homogeneous and catastrophic extremes}

One of the applications of our results is to find which covariates produce the highest tail index. This is an useful way to estimate under which conditions extremes happen, and identify regions homogeneous covariate regions which lead to the most catastrophic extremes, that is the ones having the largest tail index. In terms of extreme Kaplan--Meier integrals, we consider
\begin{align*}
    \varphi_j(x,y)=1\{x\in R_j\}
\end{align*} 
for $(R_j)$ a collection of regions in the covariate space. In this case we have that $S_{k,n}(\varphi_j)$  converges to $\pr(X^{\circ}\in R_j)$. Thus, $S_{k,n}(\varphi_j)$ will inevitably go to zero when $R_j$ does not contain the maximal tail index. 
In Figure \ref{fig:hom_dis} we investigate the regions $R_1=(1.8,2.2)$ and $R_2=(1.0,1.4)$. In addition to the EKMI estimator, we have included a naive estimator to benchmark our results, given by
\begin{align*}
N_j(k)=\frac{\#\text{$Z_i$ larger than $k$ and $X_i\in R_j$}}{\#\text{$Z_i$ larger than $k$}},\quad j=1,2,
\end{align*}
which does not take into account that the $Z_i$ are right censored, and the information coming from $\delta_i$ is discarded. Nonetheless, since the $\varphi_j$ functions do not have a $y$ component, and the dependence structure between $X$ and $C$ happens in the body of the distribution and not in the tail, we expect the two approaches to be asymptotically equivalent.
\begin{figure}[!htbp]
    \centering
    \subfigure[$R_1=(1.8,2.2)$]{
        \includegraphics[width=0.45\textwidth]{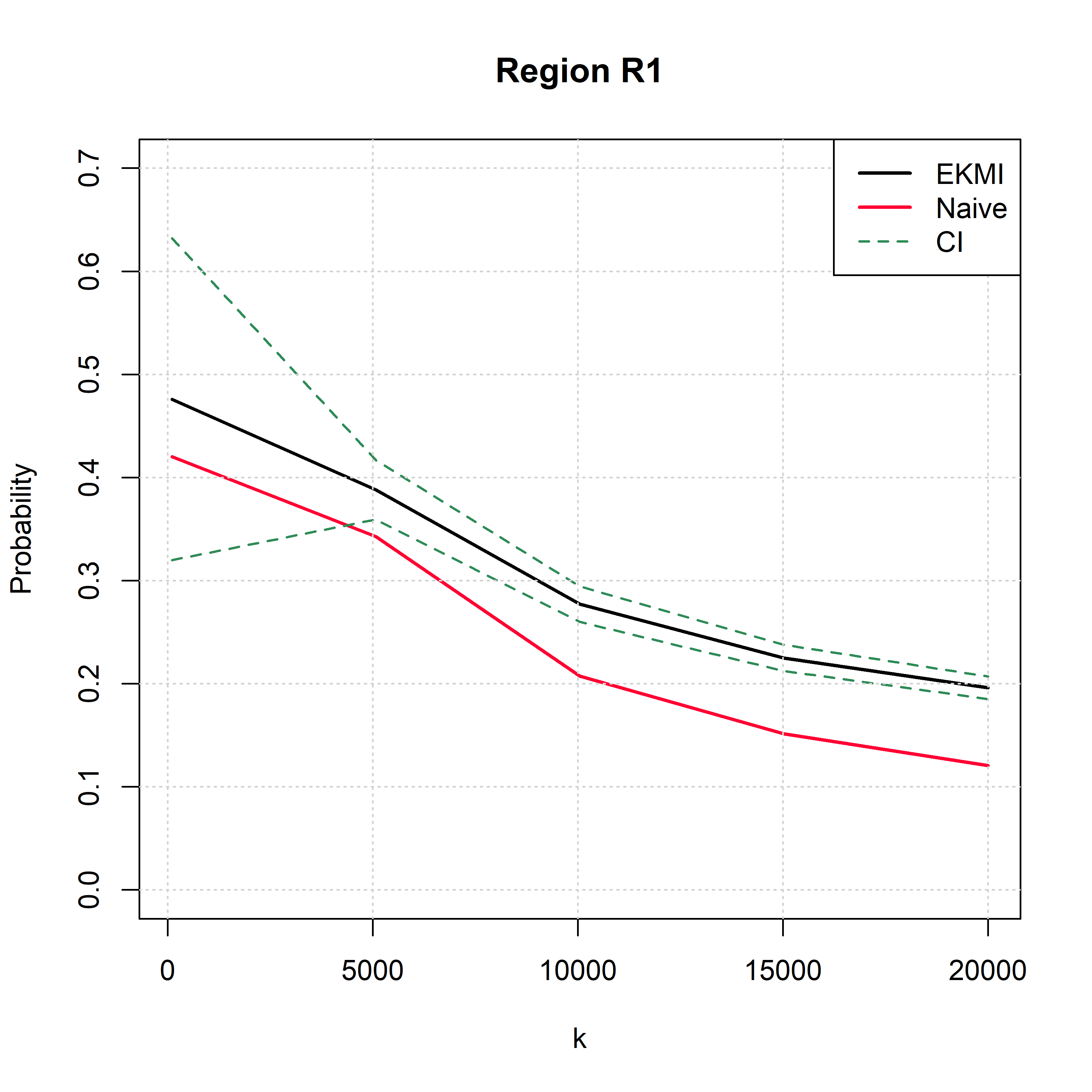}
    }
    \subfigure[$R_2=(1.0,1.4)$]{
        \includegraphics[width=0.45\textwidth]{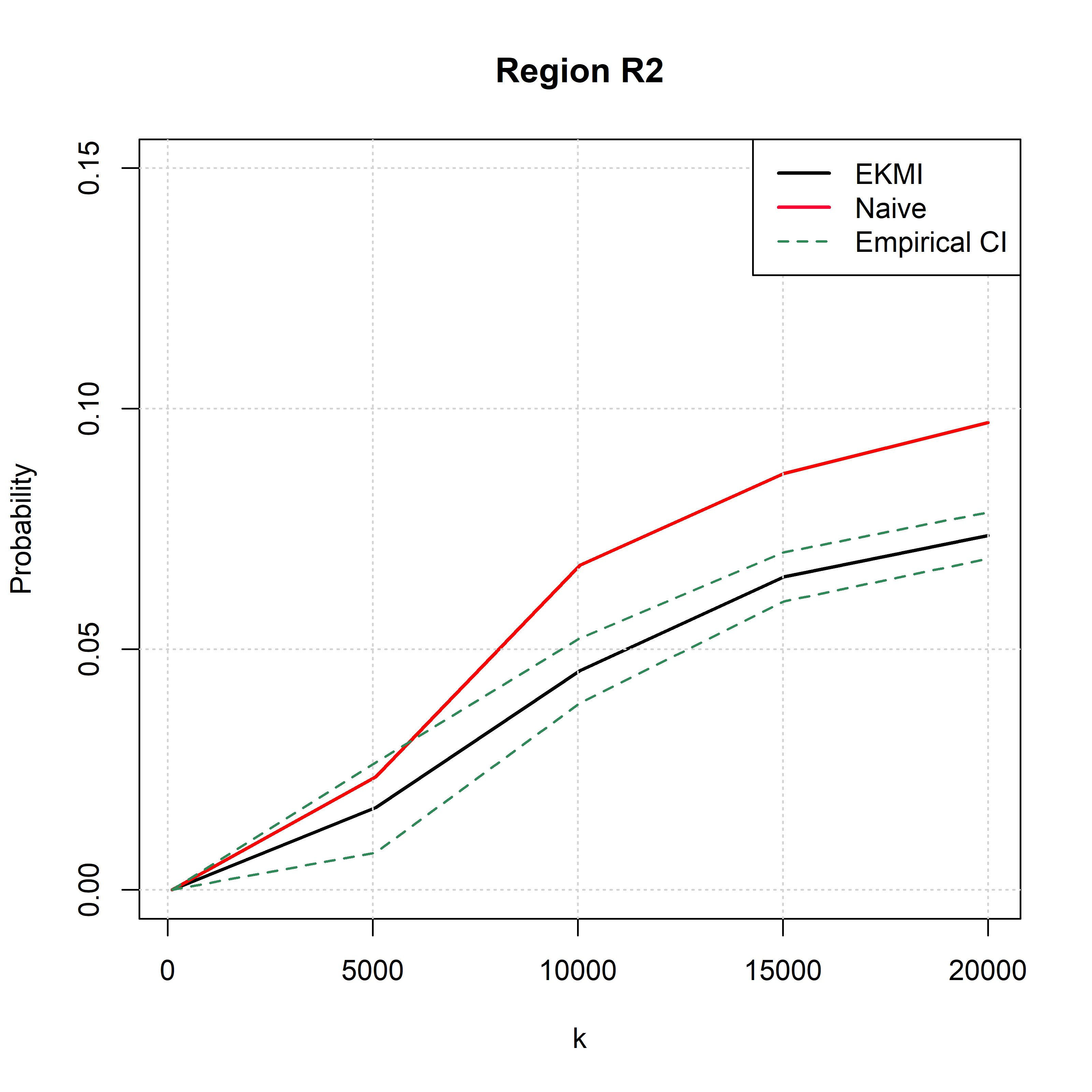}
    }
    \caption{Estimators of the asymptotic probability of  a catastrophic event belonging to the covariate regions $R_1$ (left) and $R_2$ (right). By construction, the former to converges to $0.5$, while the latter converges to $0$, as $k\to0$. In dashed lines we have the $95\%$ confidence intervals, computed through the Gaussian approximation and using plug-in estimators for the asymptotic variance.}
    \label{fig:hom_dis}
\end{figure}

Nonetheless, in Figure \ref{fig:hom_dis} we show that the EKMI estimators outperform the naive estimator in finite samples and for all quantiles, and we indeed observe how the naive estimator slightly closes the gap as $k$ becomes small. This outcome also depends on the censoring degree. For instance, if the censoring is light for small observations and heavy for large observations, then the EKMI estimators are still better performing, though to a lesser degree. In essence, by the construction mechanism of the KM product-limit estimator, a significant proportion of large values can potentially be discarded -- due to censoring -- while retaining smaller observations. In conclusion, the EKMI estimator always detects the correct covariates asymptotically, and in finite samples outperforms the naive estimator across the board, though the distance between the two estimators can vary according to the distribution of the censoring mechanism.

\subsection{Estimating the largest tail indices within regions}
\label{sec:Estimating the largest tail indices within regions}

Another useful application of EKMI estimators is to consider $\varphi$ functions which can detect not only the region where catastrophic events happen, but also the maximal tail index within that region. In other words, locally estimating the maximal tail index through the covariate space. The viewpoint of regions is in contrast to other approaches in the literature, which often focus on conditional tail index estimation. The aim of the procedure is similar to the previous subsection. However, instead of testing different regions, we estimate the tail index continuously and in that way get an understanding of how the tail index depends on the covariate. For this purpose we use a Gaussian kernel to smooth the regional effects. Concretely, we let
\begin{align*}
    \varphi^{a}_1(x,y)= \phi\left(\frac{x-a}{h}\right)\log(y),\quad \varphi^{a}_2(x,y)=\phi\left(\frac{x-a}{h}\right),
\end{align*}
where $\phi$ is a standard Gaussian kernel (density). Then the family of estimators is the given by
\begin{align*}
\hat\gamma_{n,k}(a)=\frac{S_{n,k}(\varphi^{a}_1)}{S_{n,k}(\varphi^{a}_2)}.
\end{align*}

\begin{figure}[!htbp]
    \centering
    \subfigure[\label{fig:h-0.05}]{
        \includegraphics[width=0.45\textwidth]{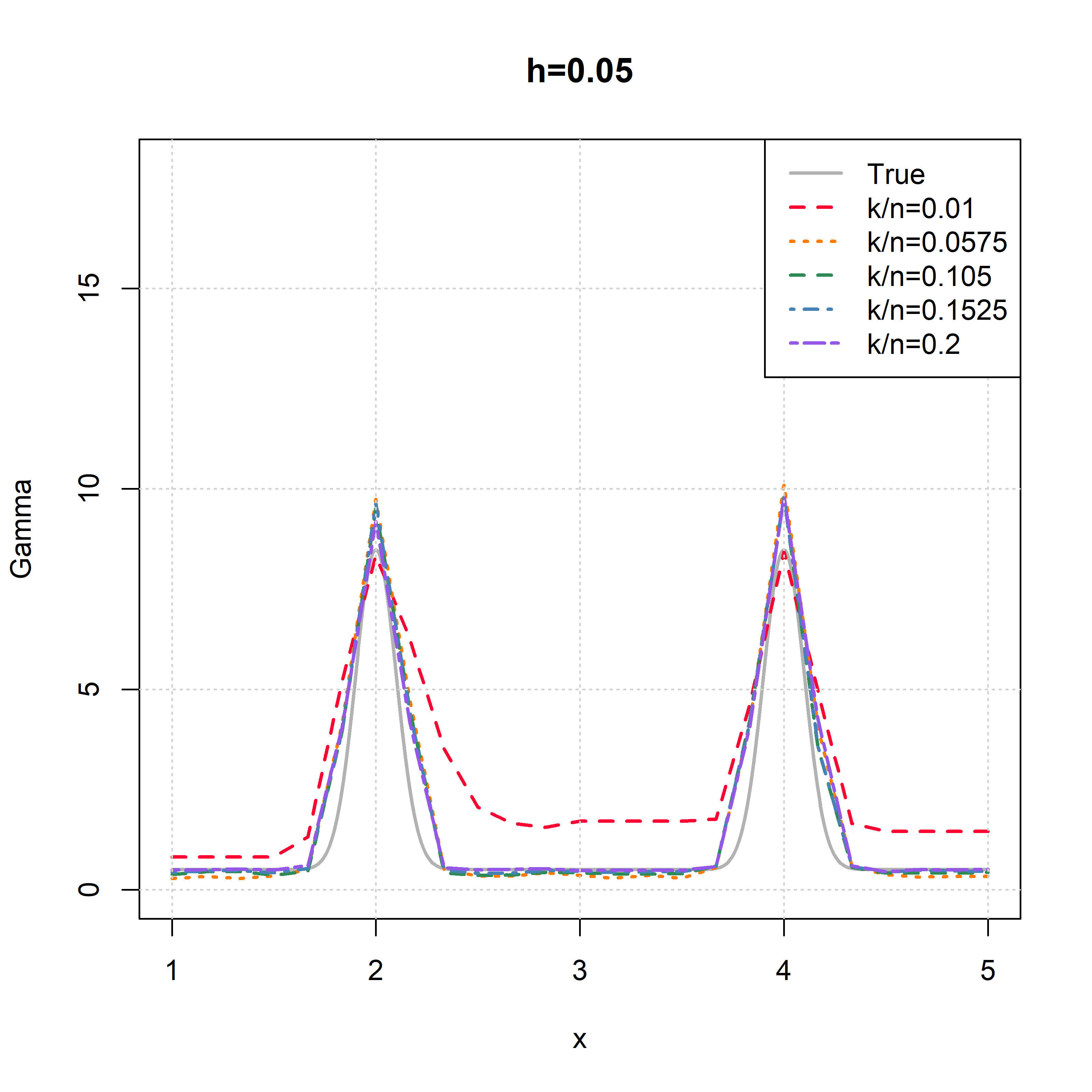}
    }
    \subfigure[\label{fig:h-0.1}]{
        \includegraphics[width=0.45\textwidth]{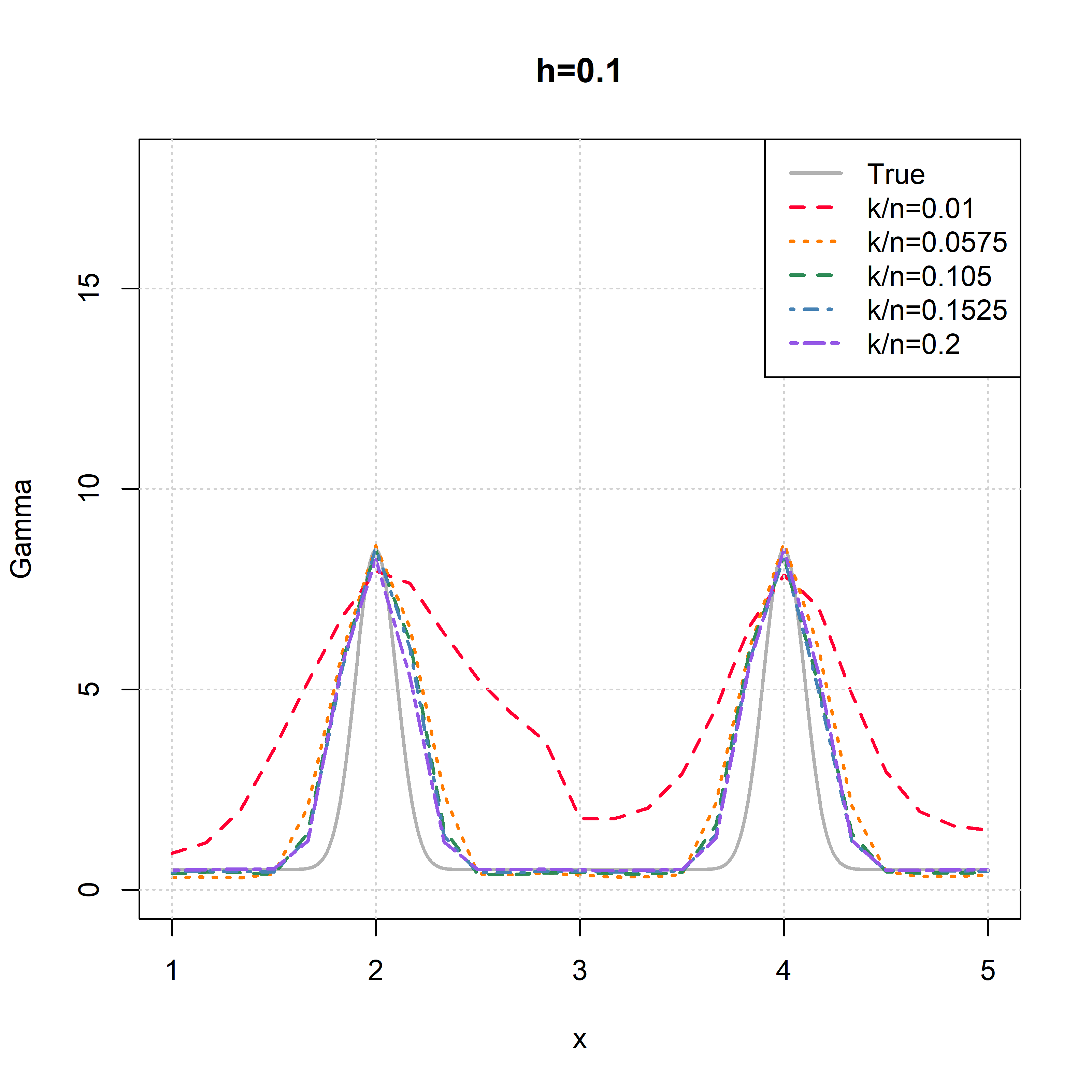}
    }
    \subfigure[\label{fig:h-0.5}]{
        \includegraphics[width=0.45\textwidth]{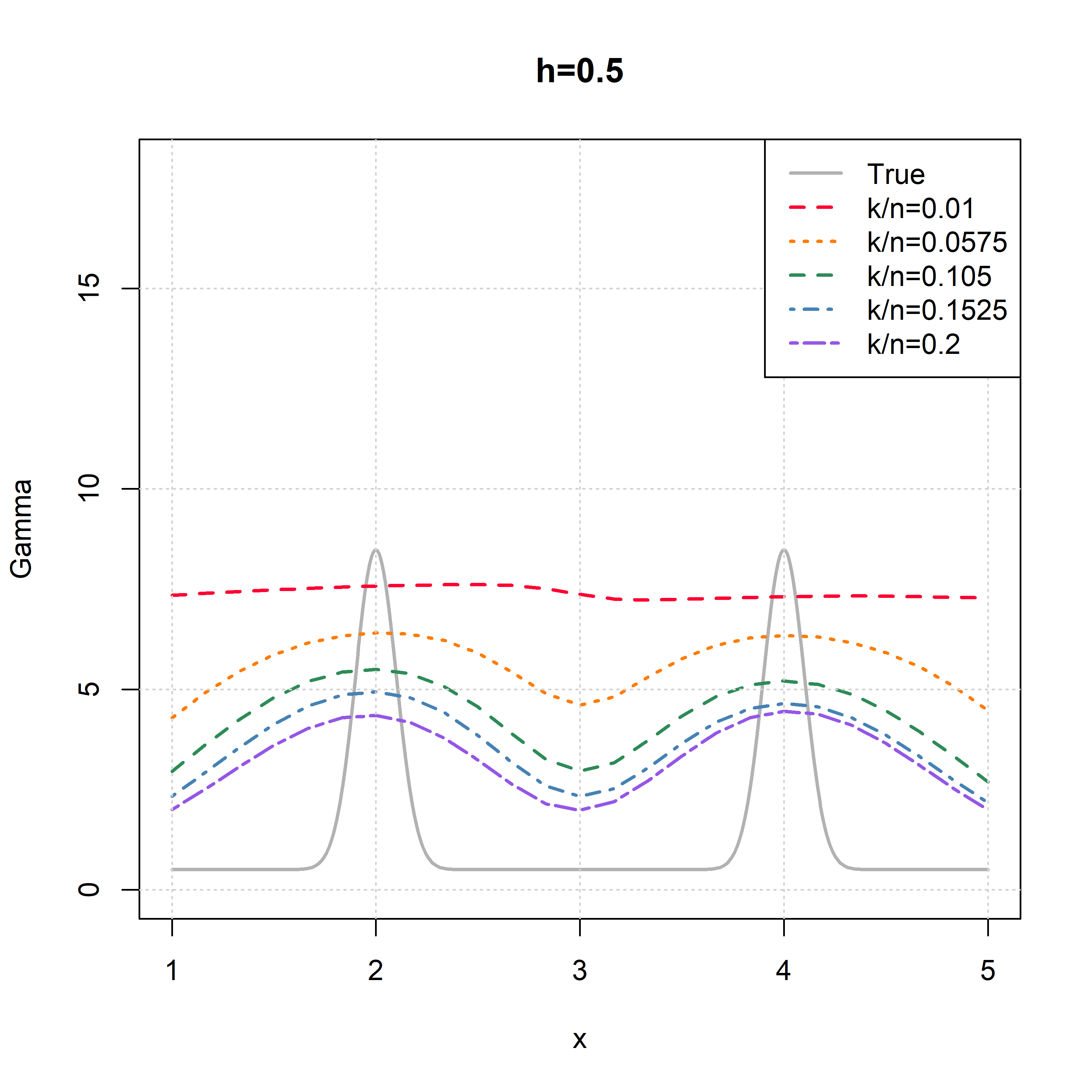}
    }
    \subfigure[\label{fig:h-opt}]{
        \includegraphics[width=0.45\textwidth]{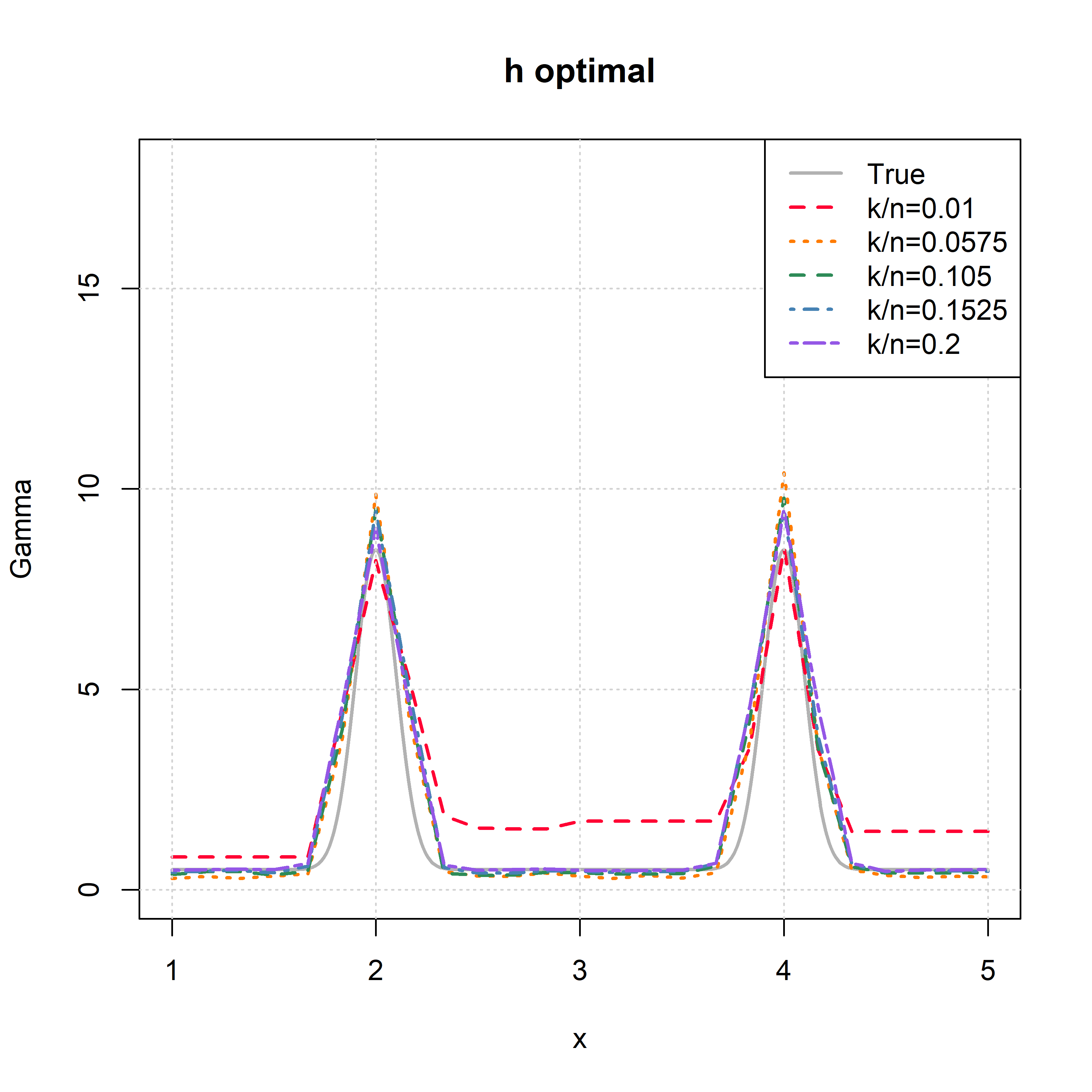}
    }
    \caption{Estimation of the tail index within different regions, using a Gaussian kernel to smooth the EMKI estimators, using three pre-selected bandwidths as well as the automatically selected bandwidth based only on covariate information.}
    \label{fig:kernel}
\end{figure}

In Figure \ref{fig:kernel} we show the estimation results in four plots corresponding to different bandwidths. In particular, we observe that for a large bandwidth  (Figure \ref{fig:h-0.5}) the estimators naturally have difficulties detecting the covariates which produces the highest tail index, however the estimated tail index still corresponds to the correct $\gamma_F^U$. Choosing a smaller bandwidth improves the detection of the regions with largest tail indices. 

Choosing $\varphi$ depending in $n$ has not been considered in the asymptotic theory, and would be of interest if the aim would be to obtain conditional tail index estimation. Since our aim is instead to describe a possible smoothing mechanism (instead of choosing regions, which is also possible), we delegate the case where $h$ depends on $n$ to future research. For sake of completeness, however, we prescribe automatic bandwidth selection according to the different standard approaches on the covariate information alone, that is, on the sample $(X_{[n-i+1,n]})_{i=1,\dots,k}$, for each $k$. This can for instance be done with the function \texttt{density} from the \texttt{R}-package \texttt{stats}. The results are depicted in Figure \ref{fig:h-opt}, where we observe that even $k/n=0.2$, produces accurate results.

\section{Real data analysis}
\label{sec:data}
This section illustrates how our methods can successfully applied to real-life data which exhibits censoring and heavy-tailed effects. Two datasets are considered, from the insurance industry, and from the a cancer study, respectively.
\subsection{Description of datasets}
\subsubsection{Insurance data}
The first dataset consists of $109,992$ observations of claim settlements of a French insurer from $1992$ to $2007$, and is publicly available through the dataset \texttt{freclaimset3dam9207} in the \texttt{R}-package \texttt{CASdatasets}. Some observations are right-censored since the claims are not fully settled. In other words, the payment of a non-settled (or open) claim might increase in subsequent years. Note that instead of taking calendar time as the response, we are taking the claim settlement (paid amount) as response. This approach has been increasingly popular for claim size modeling, see for instance the insurance data analysis of \cite{bladt2020threshold,bladt2021trimmed}.

For the analysis we concentrate on claims arriving at or before the year $2000$. This enables us to compare our estimation procedure with the fully uncensored estimates, since all the observations from before $2000$ ended up being fully settled by $2007$, which is the last observed year in the dataset. In Figures \ref{fig:fre-tail} and \ref{fig:fre-prop} we show, respectively, the estimation of the tail index (without adaptation for censoring) and the proportion of observations that are censored as a function of $k$. We observe a large bias in the tail index, though the index does not seem to vanish even for very small values of $k$. In the dataset, we have the additional covariate ``Risk Categ'', comprising of unkown risk categories and having $6$ different levels: $C1,C2,C3,C4,C5$ and $C6$. 

\subsubsection{Cancer data}

The second dataset consists of $15,564$ observations from colon cancer patients, and is also publicly available,  through the dataset \texttt{colonDC} from the \texttt{R}-package \texttt{cuRe}. This dataset pertains to a more classical domain of applications of survival analysis. We further concentrate at the subset where the covariate ``stage'' is equal to ``distant'', which leaves us with $5147$ observations. Among other things, the dataset includes follow-up time and a status indicator indicating whether a patient is alive or dead. We are interested in estimating the tail index of the lifetime of the patients, and hence follow-up time is used as our variable of interest. The data is clearly right-censored: if the patient is alive then the observation is censored. In Figures \ref{fig:colon-tail} and \ref{fig:colon-prop} we show, respectively, the estimation of the tail index and the proportion of observations that are censored as a function of $k$. Note that Condition \ref{cond-full} requires that the censoring mechanism is not too heavy (in particular less than 50\% censored observations). So we expect the EKMI estimators to become unreliable when $k$ is below $500$. As covariate information we use ``age'', indicating  the age of diagnosis for each patient. 

\begin{figure}[htbp]
    \centering
    \subfigure[Hill plot for insurance data\label{fig:fre-tail}]{
        \includegraphics[width=0.45\textwidth]{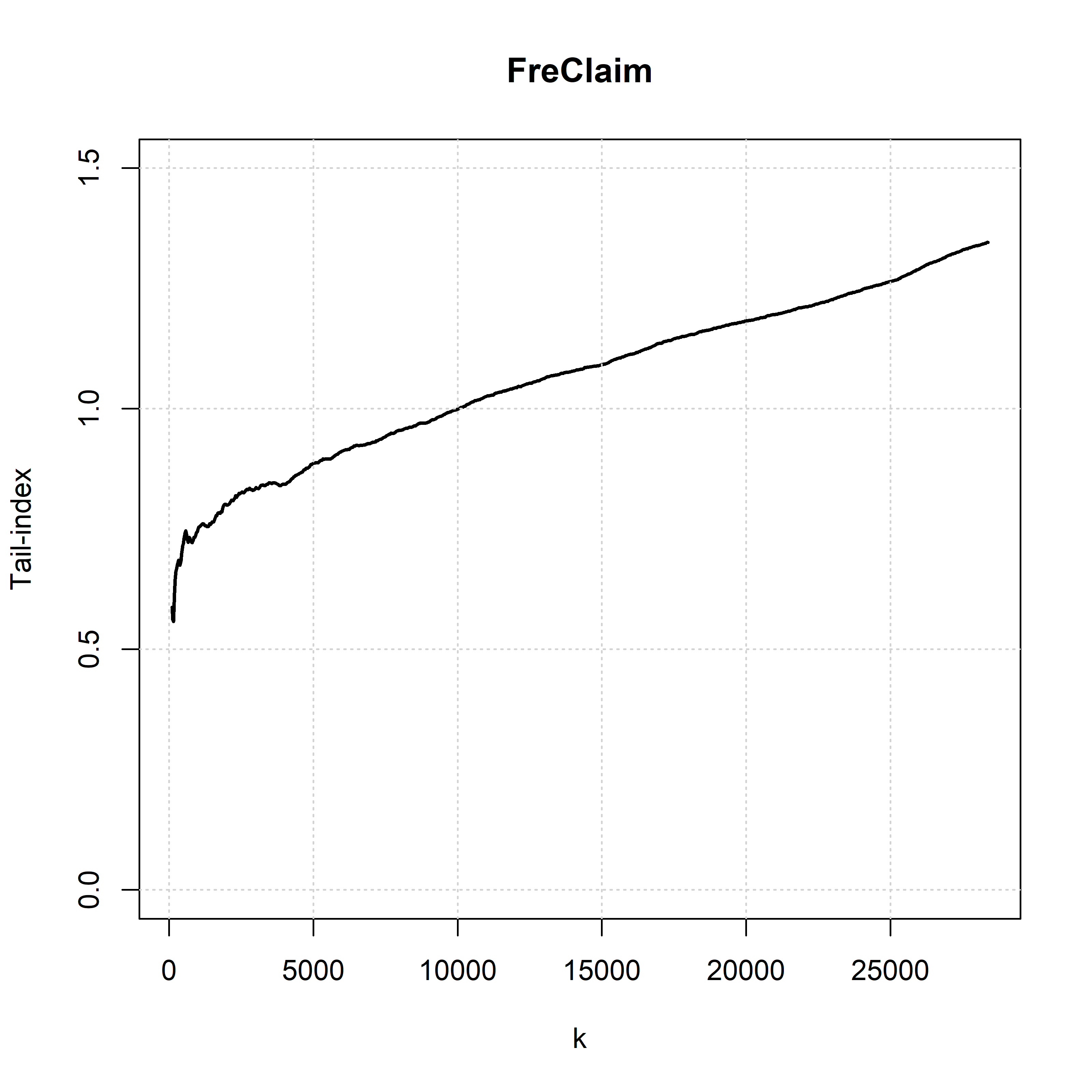}
    }
    \subfigure[Censored proportion for insurance data\label{fig:fre-prop}]{
        \includegraphics[width=0.45\textwidth]{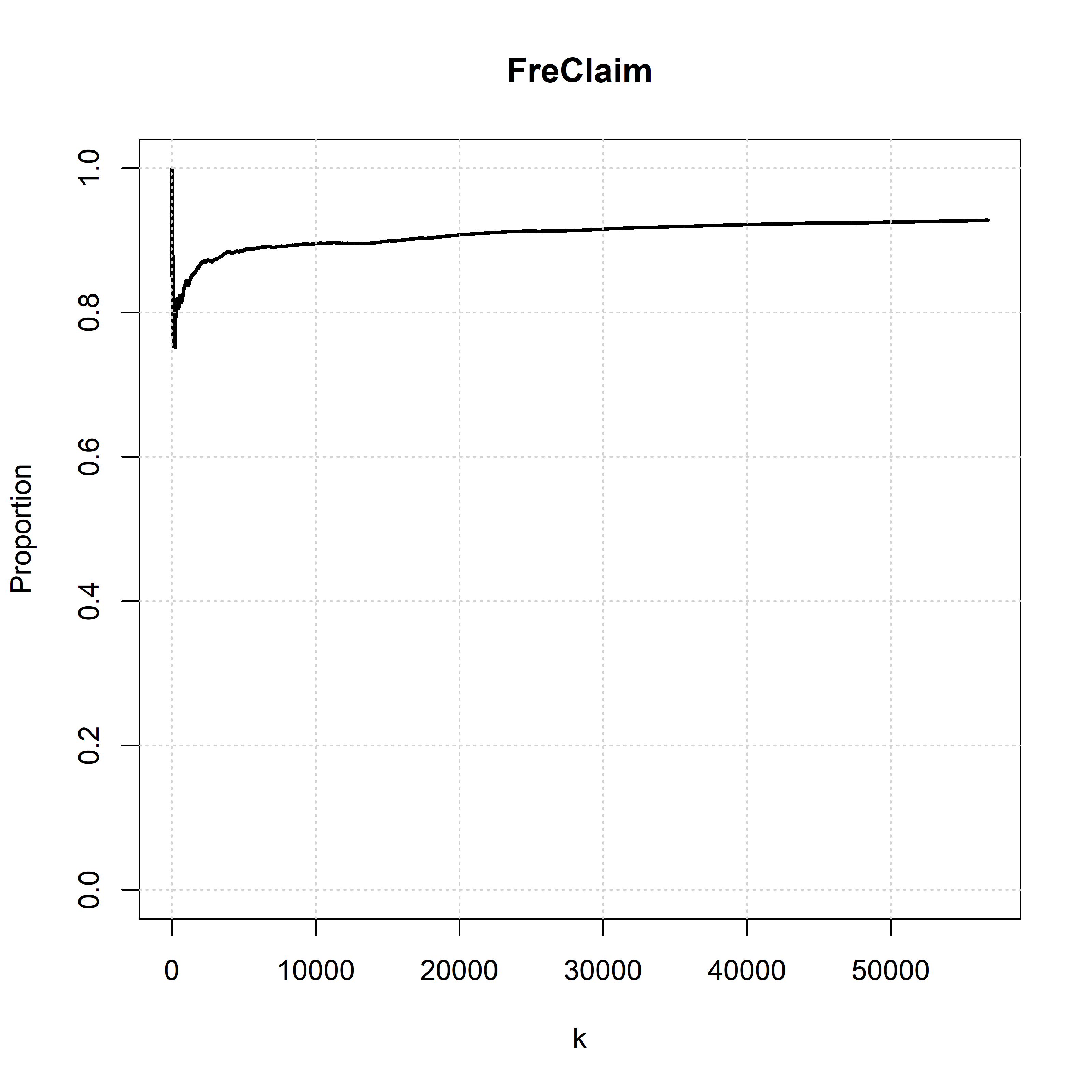}
    }
    \subfigure[Hill plot for cancer data\label{fig:colon-tail}]{
        \includegraphics[width=0.45\textwidth]{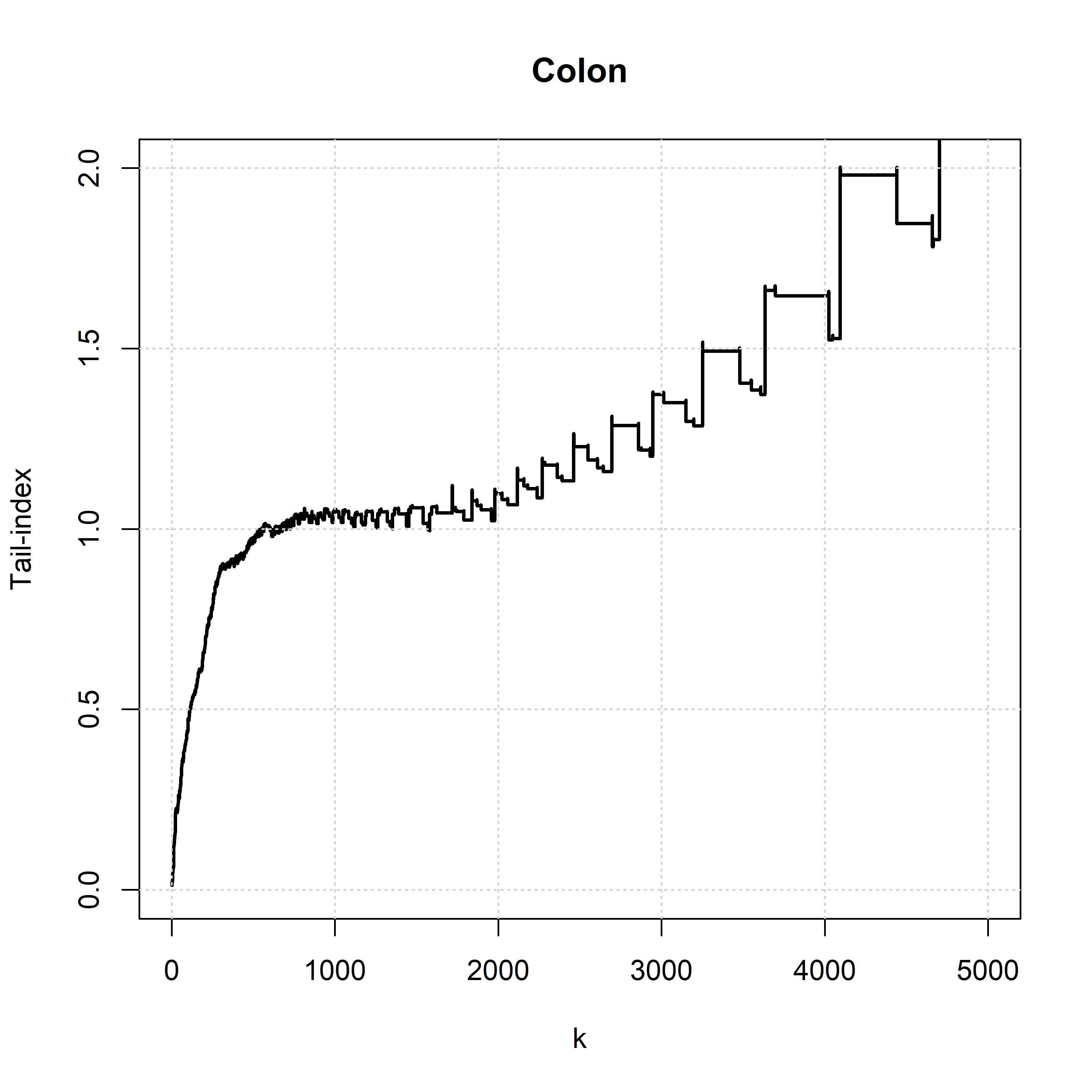}
    }
    \subfigure[Censored proportion for cancer data\label{fig:colon-prop}]{
        \includegraphics[width=0.45\textwidth]{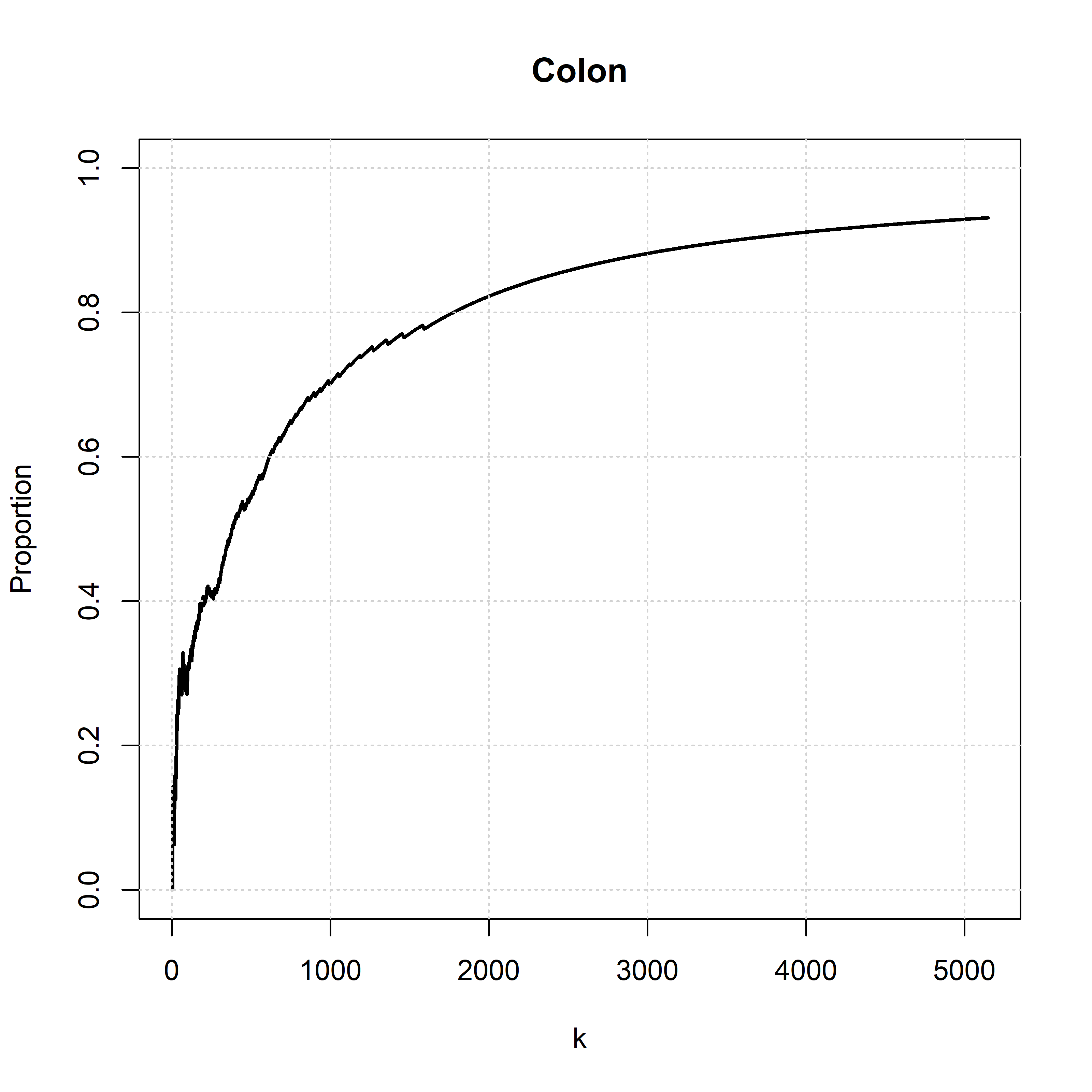}
    }
    \caption{Descriptive analysis of the two datasets, at the aggregate level, confirming a positive tail index, and a proportion of censoring mostly above $50\%$, which is required for asymptotic theory of the estimators to hold.}
    \label{fig:Fre}
\end{figure}

\subsection{Analysis: regions and tail indices}

\subsubsection{Insurance data}
For \texttt{freclaimset3dam9207} we first investigate if risk category $C1$ or $C6$ contribute most to its extremes values. This is the same setup as the simulation study of Section \ref{sec:Finding homogeneous and catastrophic extremes}. In Figure \ref{fig:fre-EKMI} we compare the EKMI estimate with the naive estimator \textit{computed using information up to 2007} (i.e. the final uncensored value of the censored observations). The two estimators are close, and it seems that both $C1$ or $C6$ contribute significantly to the occurrence of extreme values, though with a decreasing trend as $k$ decreases. Of note is that the confidence bands from the asymptotic theory seem to encompass the naive estimator for $R_1=C1$, but the $R_2=C6$ case is more contentious, suggesting that the censoring effects at lower thresholds are stronger for the latter region. Still, for smaller $k$ there is strong coherence is both cases, suggesting that the statistical model is adequate to model the evolution and censoring effects of the claim settlements. We now take into account all the risk categories, and since we have that the probabilities must sum up to $1$, we may plot the change of distribution as a function of $k$, as is done in Figure \ref{fig:fre-risk}. In terms of extreme values, it seems that $C2$ and $C3$ shrink completely for small $k$, indicating they are not risky categories. For the remaining categories, we observe remarkable stability and persistency over $k$. 

\begin{figure}[htbp]
    \centering
    \subfigure[Covariate distribution by risk category\label{fig:fre-risk}]{
        \includegraphics[width=0.45\textwidth]{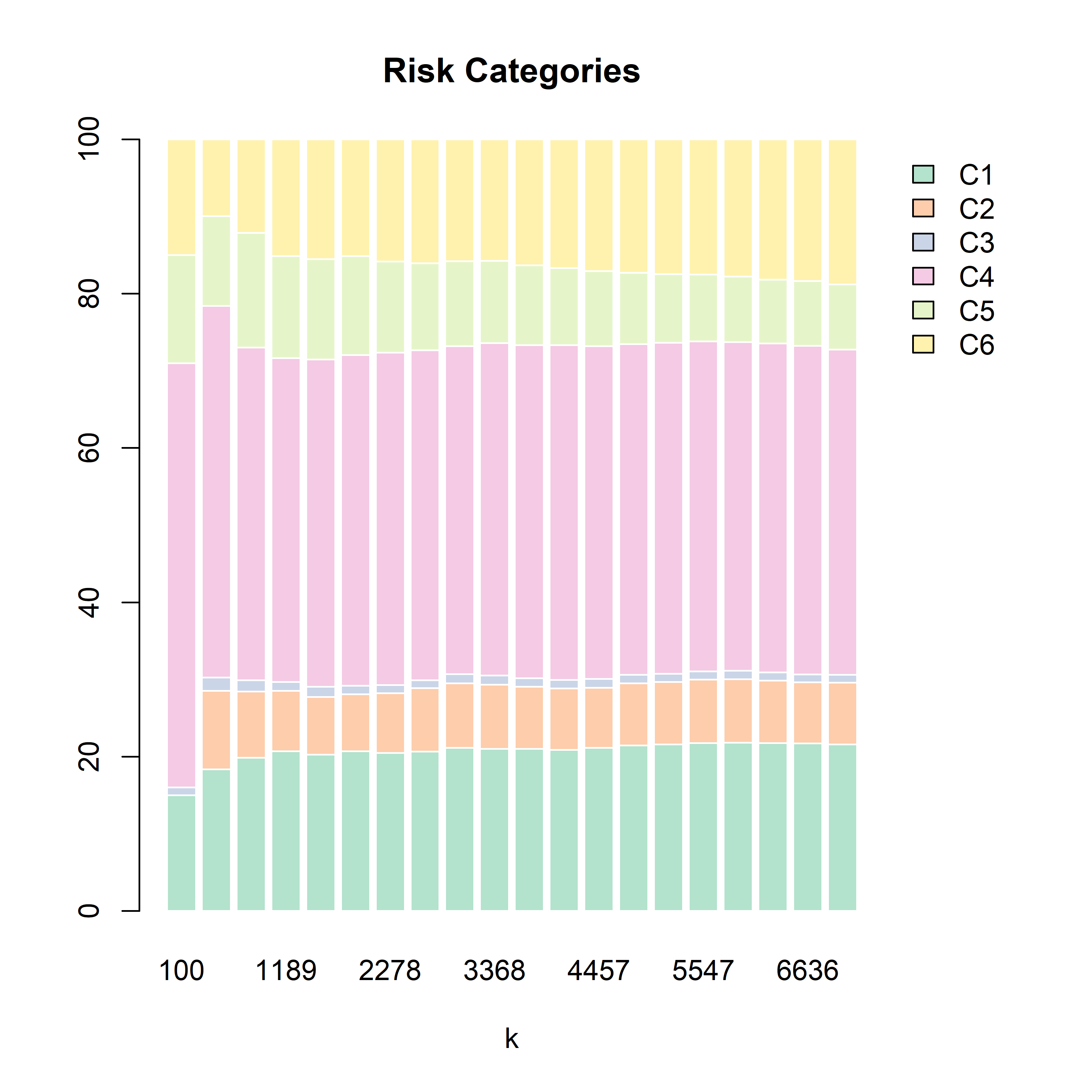}
    }
    \subfigure[EKMI versus naive estimation\label{fig:fre-EKMI}]{
        \includegraphics[width=0.45\textwidth]{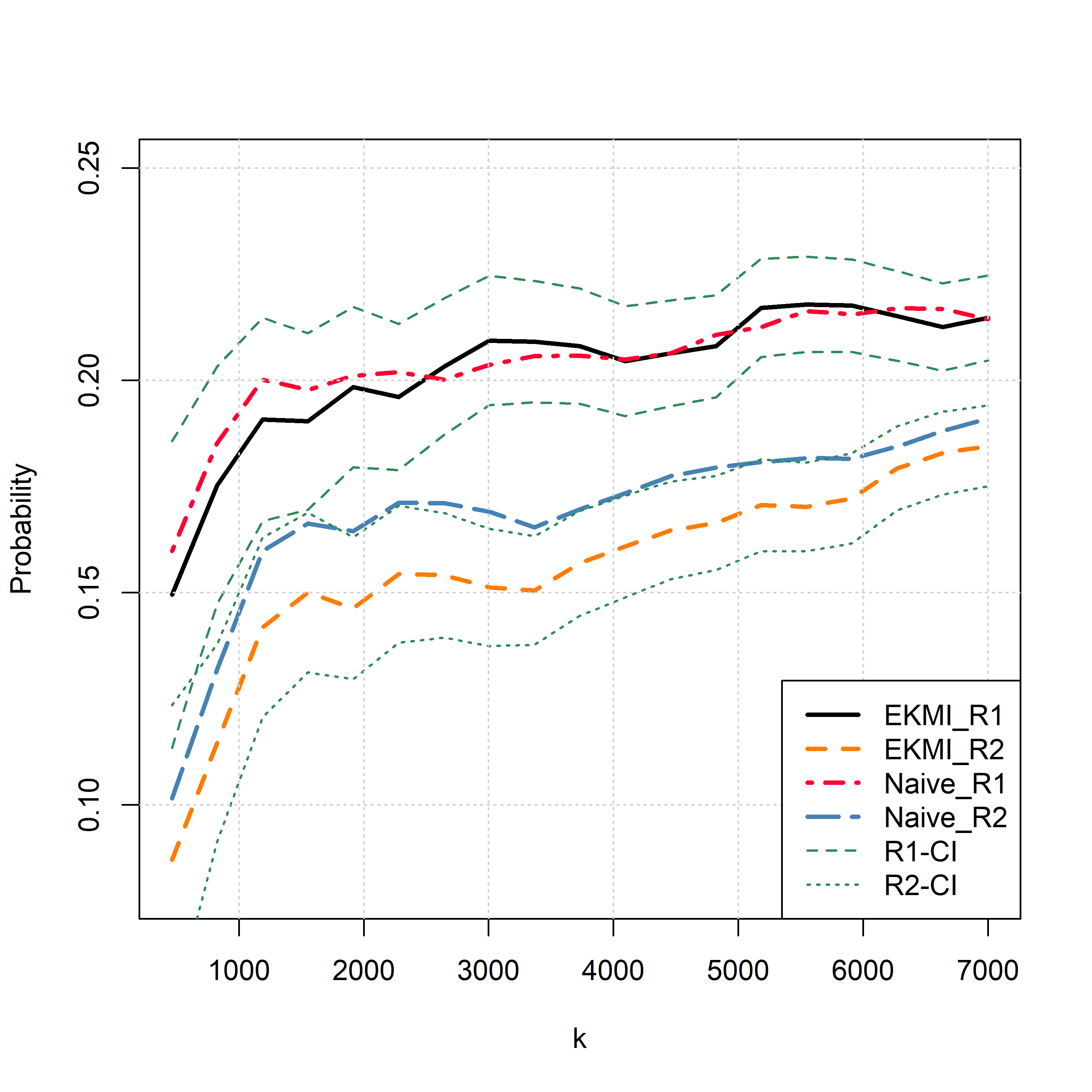}
    }
    \subfigure[Covariate distribution by age\label{fig:colon-age}]{
        \includegraphics[width=0.45\textwidth]{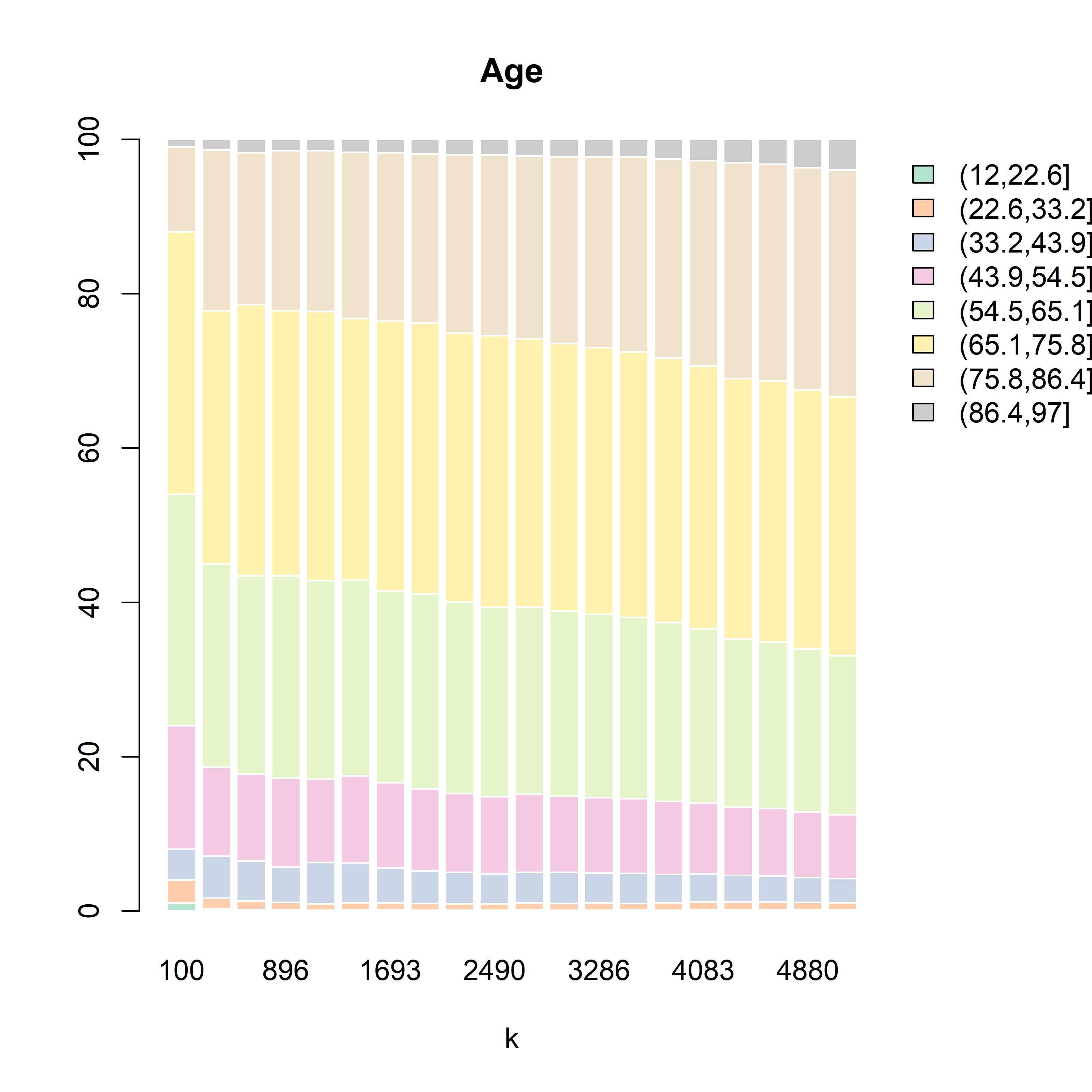}
    }
    \subfigure[Kernel smoothing of tail index\label{fig:colon-EKMI}]{
        \includegraphics[width=0.45\textwidth]{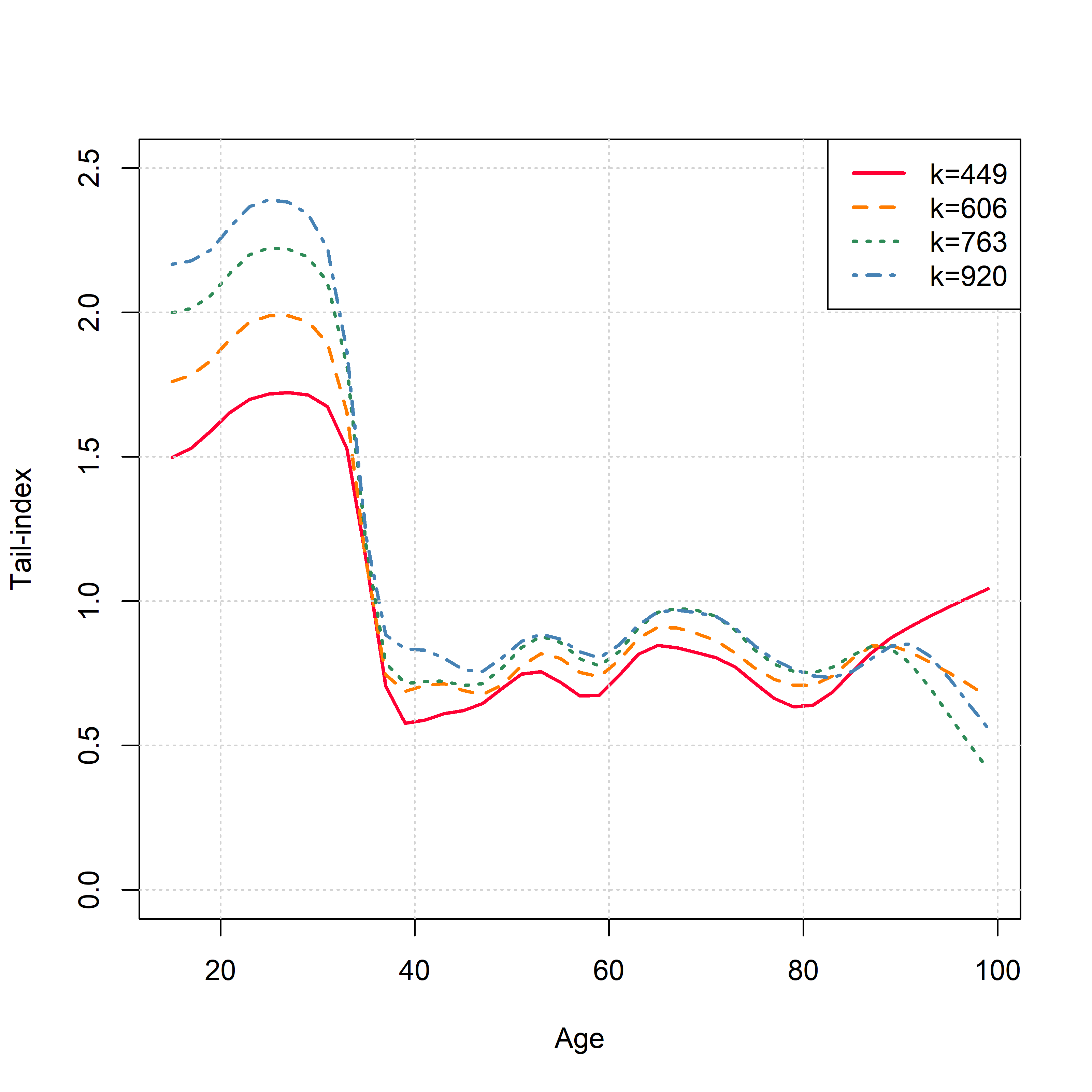}
    }
    \caption{Analysis of the insurance (top panels) and cancer (bottom panels) data.}
    \label{fig:analysis}
\end{figure}

\subsubsection{Cancer data}

For the analysis of \textit{colonDC}, we adopt the framework of Section \ref{sec:Estimating the largest tail indices within regions}, where we with to unveil the largest tail indices in different regions. We also wish to smooth the tail index for different ages through a Gaussian kernel as in the simulation study. The result can be seen in Figure \ref{fig:colon-EKMI}. We observe that patients younger than $40$ have a significant higher tail index, supporting the hypothesis that young patients get cured, so the potential time to death is significantly higher than that of older patients. Finally, we quantitatively observe in Figure \ref{fig:colon-age} that the patients with best prognosis are the youngest ones, as their covariate regions tend to expand as $k$ decreases, while the regions for older patients tend to shrink. This finding is also qualitatively in line with standard knowledge on cancer prognosis by age.

\section{Conclusion and outlook}\label{sec:conc}

This paper extends the findings of EKMI estimators by incorporating random covariates, for models in the Fréchet max-domain of attraction. The main mathematical contribution is a decomposition of the Extreme Kaplan--Meier integral along with a law of large numbers and three versions of the central limit theorem. In order to achieve said theorems, we require auxiliary uniform convergence versions of several pointwise results. To accomplish this, we have introduced natural conditions which allow for a transparent treatment with similar methods of proof as in the unconditional case. For instance, the proposed Burr model satisfies all conditions. Subsequently, we conduct simulation studies from the Burr model, where we investigate the use of EKMI estimators; we consider identification of regions where the tail-index are largest, and a smoothed estimation of maximal tail-indices across regions. Finally, we apply the two approaches to two datasets: from the insurance industry and from a cancer study, showcasing the applicability of EKMI estimators. 

The remaining two max-domains of attraction remain unexplored when covariates are present. Extending the results to those cases is difficult, though possible under further assumptions. This would give rise, for instance, to moment estimators of for tail indices across all three max-domains of attraction, when covariates are present. Finally, the connection between estimators introduced in this paper and conditional tail index models arises by considering $\varphi$ functions that are allowed to depend on the sample size $n$. Asymptotics of these objects are subject of future research.

\newpage
\bibliographystyle{imsart-nameyear} 
\bibliography{Heterogeneous.bib}       

\newpage
\appendix
\section{Uniform Potter bounds}\label{sec:potter}

We first recall the Karamata representation; namely, by Theorem B.1.6 and Remark B.1.7 in \cite{dehaan}, we have that for $\bar{F}_{Y|X=x}(y)\in RV_{-1/\gamma_F(x)}$ there exist measurable functions $\delta(x,s)$ and $c(x,t)$ such that
\begin{align*}
    \bar{F}_{Y|X=x}(y)=c(x,y)\exp\left(\int_1^y \delta(x,s)/s\,\dd s\right)
\end{align*}
where $c(x,y)\stackrel{y\to\infty}{\to} c(x) \in(0,\infty)$ and $\delta(x,y)\stackrel{y\to\infty}{\to} \delta(x)=-1/\gamma_F(x)$. 
\begin{theorem}[Uniform Potter bounds for the tail functions]
   Let $\bar{F}_{Y|X=x}(y)\in RV_{-1/\gamma_F}$ and $c$ and $a$ be uniformly convergent as $y\to\infty$. Further assume that $\inf_x c(x)>0$. Then for all $\epsilon_a>0$ and $\epsilon_c>0$ there exists $N$ such that for $t>N$ we have
    \begin{align*}
        (1-\epsilon_c)y^{-1/\gamma_F(x)-\epsilon_a}\leq \frac{\bar{F}_{Y|X=x}(ty)}{\bar{F}_{Y|X=x}(t)} \leq (1+\epsilon_c)y^{-1/\gamma_F(x)+\epsilon_a}
    \end{align*}
    for all $x\in \mathcal{X}$ and $y>1$.
    \label{thm:potter-F}
\end{theorem}

\begin{proof}
    We have that
    \begin{align*}
        \frac{\bar{F}_{Y|X=x}(ty)}{\bar{F}_{Y|X=x}(t)}=\frac{c(x,ty)}{c(x,t)}\exp\left(\int_t^{ty}\delta(x,s)/s \dd s\right).
    \end{align*}
For every $\epsilon_1>0$ there exists a $N>0$ such that $|c(y,x)-c(x)|<\epsilon_1$ for $y\ge N$, and for all $x\in \mathcal{X}$. Then for every $\epsilon_c>0$ we can choose a $N$ such that $\epsilon_1\leq \frac{c_* \epsilon_c}{2+\epsilon_c}$, where $c_*=\inf_x c(x)$. Then for $t>N$, we have
\begin{align*}
    \frac{c(yt,x)}{c(t,x)}\leq \frac{c(x)+\epsilon_1}{c(x)-\epsilon_1}<1+\epsilon_c.
\end{align*}
Likewise we have for $\epsilon_1<c_*\epsilon_c/(2-\epsilon_c)$ that
\begin{align*}
    \frac{c(yt,x)}{c(t,x)}\geq \frac{c(x)-\epsilon_1}{c(x)+\epsilon_1}>1-\epsilon_c.
\end{align*}
Now look at $\exp\left(\int_t^{ty}\delta(x,s)/s\,\dd s\right)$. For large enough $N$, we have that for $s\geq N$ then $|\delta(x,s)-\delta(x)|<\epsilon_a$ for all $\epsilon_a>0$. So for $y\geq N$

\begin{align*}
    \exp\left(\int_t^{ty}(\delta(x)-\epsilon_a)/sds\right)&\leq \exp\left(\int_t^{ty}\delta(x,s)/sds\right)\\
    &\leq \exp\left(\int_t^{ty}(\delta(x)+\epsilon_a)/sds\right) 
\end{align*}
which is equivalent to $$    y^{-1/\gamma_F(x)-\epsilon_a}\leq \exp\left(\int_t^{ty}\delta(x,s)/s\,\dd s\right) \leq y^{-1/\gamma_F(x)+\epsilon_a}.$$
Combining both terms, we have the desired result for $y>N$ and $N$ large enough.
\end{proof}

We now establish uniform Potter bounds for $U_{F_{Y|X=x}}$. First note that if $L$ is slowly varying, then there exists a slowly varying function $L^*$, which satisfies
\begin{align}
    L(x)L^*(xL(x))\rightarrow 1
    \label{eq:deBruyn}
\end{align}
as $x\rightarrow \infty$. $L^*$ is called the De Bryun conjugate of $L$ (See proposition 2.5 in \cite{beirlant2006statistics} and the following discussion for more information).

\begin{theorem}[Uniform Potter bounds for the tail quantile functions]
    Assume the conditions of Theorem \ref{thm:potter-F}. In addition assume that $U_{F_{Y|X=x}}$ is continuous for each $x\in\mathcal{X}$ and that there exist a $K>0$ such that for all $x\in \mathcal{X}$ we have for $y>K$ that $L_{U_{F_{Y|X=x}}}(y)$ is monotone and $L_{F_{Y|X=x}}(y)$ is continuous. Finally, assume that \eqref{eq:deBruyn} holds uniformly in $x\in\mathcal{X}$ for $L_{F_{Y|X=x}}(y)^{-\gamma(x)}$. Then for all $\epsilon_a>0$ and $\epsilon_c>0$ there exists a $T$ such that for $t>T$ we have
    \begin{align*}
        (1-\epsilon_c)y^{\gamma_F(x)-\epsilon_a}\leq \frac{U_{F_{Y|X=x}}(ty)}{U_{F_{Y|X=x}}(t)} \leq (1+\epsilon_c)y^{\gamma_F(x)+\epsilon_a}
    \end{align*}
    for all $x\in \mathcal{X}$ and $y>1$.
    \label{thm:potter-U}
\end{theorem}

\begin{proof}
    The proof for the upper bound is shown. The lower bound follows by similar argumentation. So for $A>1$, we show that
    \begin{align*}
        \frac{U_{F_{Y|X=x}}(tA)}{U_{F_{Y|X=x}}(t)}\leq(1+\epsilon_c)A^{\gamma_F(x)+\epsilon_a}
    \end{align*}
    for some $\epsilon_c,\epsilon_a$ and $t> T$.
    It is equivalent to show
    \begin{align*}
        \frac{L_{U_{F_{Y|X=x}}}(tA)}{L_{U_{F_{Y|X=x}}}(t)}\leq(1+\epsilon_c)A^{\epsilon_a}.
    \end{align*}
First assume that $L_{U_{F_{Y|X=x}}}(y)$ is monotone decreasing for $y>K$. It is possible for all $A>1$ to be written as $A=((1+\epsilon_{c_1})y^{1+\epsilon_{a_1}})^{1/\gamma(x)}$ for some $\epsilon_{c_1}$, $\epsilon_{a_1}>0$ and $y>1$. Further write $t=t_1L_x(t_1)$,  where $L_x:=L_{F_{Y|X=x}}^{-\gamma(x)}$. This implies that there exist uniform Potter bounds for $L_x$, and hence there exist a $T_1>0$ such that for $t_1>T_1$ we have 
    \begin{align*}
        ((1+\epsilon_{c_1})y^{\epsilon_{a_1}})^{-1}\leq \frac{L_x(t_1)}{L_x(t_1y)}.
    \end{align*}
   Since ${U_{F_{Y|X=x}}}$ is continuous, we can conclude from the discussion in Section 2.9.3 in \cite{beirlant2006statistics} that $L_{U_{F_{Y|X=x}}}(y)=L_x^*(y^{\gamma(x)})$, where $L_x^*$ is the de Bryun conjugate of $L_x$. Now we have for $t_1>\max\{K,T_1\}$
   \begin{align*}
       \frac{L_{U_{F_{Y|X=x}}}(tA)}{L_{U_{F_{Y|X=x}}}(t)}&=\frac{L_x^*((1+\epsilon_{c_1})y^{1+\epsilon_{a_1}}t_1L_x(t_1))}{L_x^*(t_1L_x(t_1))} \\
       &=\frac{L_x^*\left((1+\epsilon_{c_1})y^{1+\epsilon_{a_1}}t_1L_x(yt_1)\frac{L_x(t_1)}{L_x(t_1y)}\right)}{L_x^*(t_1L_x(t_1))} \\
       & \leq \frac{L_x^*(t_1yL_x(yt_1))}{L_x^*(t_1L_x(t_1))}.
   \end{align*}
   Since the uniform de Bryun conjugate of $L_x$ is $L_x^*$, we have that 
   \begin{align*}
       \frac{L_x^*(t_1yL_x(yt_1))}{L_x^*(t_1L_x(t_1))}\sim\frac{L_x(t_1)}{L_x(t_1y)}
   \end{align*}
    and hence for all $\epsilon_1>0$ there exist some $T_2>0$, such that for $t_1>T_2$ we have
    \begin{align*}
        \frac{L_x^*(t_1yL_x(yt_1))}{L_x^*(t_1L_x(t_1))}\leq \frac{L_x(t_1)}{L_x(t_1y)}+\epsilon_1
    \end{align*}
    By definition of $L_x$, we have
    \begin{align*}
        \frac{L_x(t_1)}{L_x(t_1y)}+\epsilon_1\leq \left(\frac{L_{F_{Y|X=x}}(t_1y)}{L_{F_{Y|X=x}}(t_1)}\right)^{\gamma_F(x)}+\epsilon_1.
    \end{align*}
    By using the uniform Potter bounds for $L_{F_{Y|X=x}}$, there exists some $T_3>0$, such that for $\epsilon_{c_2},\epsilon_{a_2}>0$, we have for $t_1>T_3$
    \begin{align*}
        \left(\frac{L_{F_{Y|X=x}}(t_1y)}{L_{F_{Y|X=x}}(t_1)}\right)^{\gamma_F(x)}+\epsilon_{1}&\leq (1+\epsilon_{c_2})^{\gamma^U}y^{\epsilon_{a_2}\gamma^U}+\epsilon_1 \\
        &=(1+\epsilon_{c_2})^{\gamma^U}\frac{A^{\gamma^U/(1+\epsilon_{a_1})\epsilon_{a_2}\gamma^U}}{(1+\epsilon_{c_1})^{\epsilon_{a_2}\gamma^U/(1+\epsilon_{a_1})}}+\epsilon_1 \\
        &\leq (1+\epsilon_c)A^{\epsilon_a}
    \end{align*}
    for some arbitrarily small $\epsilon_a,\epsilon_c>0$ for a large enough $t_1$, and hence we can conclude that there exist a $T>0$ such that for $t>T$ the desired result holds.
    
    The rewriting of $t$ and $A$ might depend on $x$.
    The latter is not a problem since $y>1$ can be chosen freely and does not change the value of $T_1$. Regarding the former quantity, we need to show that for all $x$, $t$ can be represented as $t=t_1L_x(t_1)$ for $t_1>\max\{T_1,T_2, T_3, K\}$. Clearly $t_1L_x(t_1)\rightarrow \infty$. Furthermore, since $L_x(t_1)$ is continuous for $t_1>K$, we have that $t_1L_x(t_1)$ is continuous for $t_1>K$. Now let $T_4=\max\{T_1,T_2,T_3, K\}$. If we can select $t$ larger than $T_4L_x(T_4)$ then there exists $t_1>T_4$ such that $t=t_1L_x(t_1)$ for all $x\in \mathcal{X}$. 
    
    Finally, we show that it is possible to select such a $t$. I.e. we need to show that $T_4L_x(T_4)$ is bounded in $\mathcal{X}$. By the Karamata representation we have that for some $T_5>0$ then for all $y>T_5$ there exist $0<\epsilon_1<c$ and $\epsilon_2>0$ such that
    \begin{align*}
        L_{F_{Y|X=x}}(y)\geq (c-\epsilon_1)y^{-\epsilon_2}
    \end{align*}
    where $c=\inf_{x\in \mathcal{X}} c(x)>0$. So
    \begin{align*}
        L_x\leq \left(\frac{1}{c-\epsilon_1}\right)^{\gamma(x)}y^{\epsilon_2\gamma(x)}\leq K_1 y^{\gamma^U\epsilon_2}
    \end{align*}
    Hence by choosing $t>\max\{K_1 T_5^{1+\gamma^U\epsilon_2},K_1 T_4^{1+\gamma^U\epsilon_2}\}$ the upper Potter bound holds uniformly. 

    By an analogous argument we we may treat the case of when $L_{F_{Y|X=x}}$ is monotone increasing where we let $A=(1-\epsilon_c)y^{-\epsilon_a}$. Finally, the lower bound can be shown similarly.
\end{proof}

\begin{remark}
    Theorem \ref{thm:potter-U} still holds if we instead assume similar conditions as in Theorem \ref{thm:potter-F}.
\end{remark}

\newpage
\section{Uniform Segers bounds}\label{sec:segers}

This section is devoted to the uniform extensions of Lemma 4.3 and Lemma 4.6 in \cite{segers}.

\begin{lemma}[Extended Segers bounds ]
\label{segers}
    Assume that $U_{F_{Y|X=x^0}}$ is uniform normalized. There exists a $\epsilon>0$, such that for $t,s>T$ for some $T>0$ and $y>1$
    \begin{align}
        y^{\gamma_F(x^0)-\epsilon}\leq\frac{U_{F_{Y|X=x^0}}(sy)}{U_{F_{Y|X=x^0}}(s)}\leq y^{\gamma_F(x^0)+\epsilon}
        \label{Segers1}
    \end{align}
    and
    \begin{align}
        &\left|\frac{U_{F_{Y|X=x^0}}(sy)}{U_{F_{Y|X=x^0}}(s)}-\frac{U_{F_{Y|X=x^1}}(ty)}{U_{F_{Y|X=x^1}}(t)}\right|\nonumber\\
        &\leq2\left(\int_s^{\infty} |\eta_{x^0}'(l)|\dd l+\int_t^{\infty} |\eta_{x^0}'(l)|\dd l\right) |\log(s/t)|y^{\gamma^U+\epsilon}\nonumber\\
        &\quad+\left(\int_t^{\infty} |\eta_{x^0}'(l)|+|\eta_{x^1}'(l)||\dd l+|\eta_{x^0}-\eta_{x^1}|\right)\log(y)y^{\gamma^U+\epsilon}\label{Segers2}
    \end{align}
    for all $x_0,x_1\in \mathcal{X}$
\end{lemma}

\begin{proof}

    We start by showing \eqref{Segers1}. Since $U_{F_{Y|X=x^0}}$ is uniform normalized, we have that for $\epsilon>0$ we can find a $T$, such that for $s>T$, we have $|\eta_x(s)-\eta_x|<\epsilon$ for all $x\in\mathcal{X}$. Then
    \begin{align*}
        \frac{U_{F_{Y|X=x^0}}(sy)}{U_{F_{Y|X=x^0}}(s)}&=\exp\left(\int_s^{sy}\eta_{x^0}(u)/u\dd u\right)\\
        &\leq \exp\left((\epsilon+\gamma_F(x^0)\int_s^{sy}1/u \dd u\right)\\
        &\leq y^{\gamma_F(x^0)+\epsilon} 
    \end{align*}
    The lower bound can be obtained similarly. We now turn to \eqref{Segers2}. Note that
    
    \begin{align*}
        &\left|\frac{U_{F_{Y|X=x^0}}(sy)}{U_{F_{Y|X=x^0}}(s)}-\frac{U_{F_{Y|X=x^1}}(ty)}{U_{F_{Y|X=x^1}}(t)}\right|\\
        &=\left| \exp\left(\int_s^{sy}\eta_{x^0}(u)/u\dd u\right)-\exp\left(\int_t^{ty}\eta_{x^1}(u)/u \dd u \right)\right| \\
        &\leq \left|\int_s^{sy}\eta_{x^0}(u)/u \dd u-\int_t^{ty}\eta_{x^1}(u)/u \dd u \right| \max\left\{\frac{U_{F_{Y|X=x^0}}(sy)}{U_{F_{Y|X=x^0}}(s)},\frac{U_{F_{Y|X=x^1}}(ty)}{U_{F_{Y|X=x^1}}(t)}\right\}.        
    \end{align*}
Both expressions in the maximum are less or equal to $y^{\gamma^U+\epsilon}$ by \eqref{Segers1}. Again, for all $\epsilon>0$, there exists a $T$ such that for $u>T$, we have $|\eta_{x}(u)-\eta_{x}|<\epsilon$ for all $x\in \mathcal{X}$.  Note that 
    \begin{align*}
        |\eta_x(u)-\eta_x|=\left|\int_u^{\infty} \eta_x'(l)\dd l\right|\leq \int_u^{\infty} |\eta_x'(l)|\dd l
    \end{align*}
where the upper bound is decreasing in u. So for $u\geq\min\{s,t,\}$ and $y>1$, we have
    \begin{align*}
        |\eta_x(uy)-\eta_x(u)|\leq2\left(\int_s^{\infty} |\eta_x'(l)|\dd l+\int_t^{\infty} |\eta_x'(l)|\dd l\right).
    \end{align*}
    In addition, we have
    \begin{align*}
        |\eta_{x^0}(u)-\eta_{x^1}(u)|&\leq|\eta_{x^0}(u)-\eta_{x^0}|+|\eta_{x^1}(u)-\eta_{x^1}|+|\eta_{x^0}-\eta_{x^1}| \\
        &\leq \int_t^{\infty} |\eta_{x^0}'(l)|+|\eta_{x^1}'(l)||\dd l+|\eta_{x^0}-\eta_{x^1}|.
    \end{align*}
    Consequently
    \begin{align*}
        &\left|\int_s^{sy}\eta_{x^0}(u)/u \dd u-\int_t^{ty}\eta_{x^1}(u)/u \dd u \right|\\
        &=\left|\int_s^{sy}\eta_{x^0}(u)/u \dd u-\int_t^{ty}\eta_{x^0}(u)/u \dd u+\int_t^{ty}\eta_{x^0}(u)/u \dd u-\int_t^{ty}\eta_{x^1}(u)/u \dd u \right| \\
        &\leq\left|\int_s^{sy}\eta_{x^0}(u)/u \dd u-\int_t^{ty}\eta_{x^0}(u)/u \dd u\right|\\
        &\quad+\left|\int_t^{ty}\eta_{x^0}(u)/u \dd u-\int_t^{ty}\eta_{x^1}(u)/u \dd u \right| \\
        &=\left|\int_t^{s}\frac{\eta_{x^0}(uy)-\eta_{x^0}(u)}{u}\dd u \right|+\left|\int_t^{ty}\frac{\eta_{x^0}(u)-\eta_{x^1}(u)}{u} \dd u\right| \\
        &\leq 2\left(\int_s^{\infty} |\eta_{x^0}'(l)|\dd l+\int_t^{\infty} |\eta_{x^0}'(l)|\dd l\right) \log(s/t)\\
        &+\left(\int_t^{\infty} |\eta_{x^0}'(l)|+|\eta_{x^1}'(l)||\dd l+|\eta_{x^0}-\eta_{x^1}|\right)\log(y).
    \end{align*}
\end{proof}
\begin{lemma}[Second-order uniform Segers bound]
    Let $U_{F_{Y|X}}$ satisfy Condition \ref{cond:second-order}. Assume $\max_x \rho(x)<0$ if ${X}$ is continuous. Then for all $\epsilon>0$ there exist $T>1$ and $K>0$ such that
    \begin{align*}
        |U_{F_{Y|X=x}}(ty)/U_{F_{Y|X=x}}(t)-y^{\gamma_F(x)}|\leq Ky^{\gamma_F(x)} a_{X=x}(t)
    \end{align*}
    for $t\geq T$, $y\geq 1$ and all $x\in \mathcal{X}$.
    \label{lemma:segers2}
\end{lemma}

\begin{proof}
    Note that $\frac{y^{\rho(x)}-1}{\rho(x)}$ 
    has upper bound $-1/\rho(x)$. Hence for all $\epsilon>0$ and $y\geq 1$
    \begin{align*}
        \frac{y^{\rho(x)}-1}{\rho(x)}\leq -1/\rho^*
    \end{align*}
    where $\rho^*=\max_x \rho(x)$. Since we have uniform convergence, we can conclude that for large $t$, there exists an $\epsilon_1>0$ such that
    \begin{align*}
        \frac{U_{F_{Y|X=x}}(ty)/U_{F_{Y|X=x}}(t)-y^{\gamma_F(x)}}{a_{X=x}(t)}&\leq(1+\epsilon_1)y^{\gamma_F(x)}\frac{y^{\rho(x)}-1}{\rho(x)}\leq K y^{\gamma_F(x)}.
    \end{align*}
    This concludes the proof.
\end{proof}

\newpage
\section{Proof of main results}
\label{sec:proofs-main}

Most proof techniques are inspired by \cite{EKM,stute,stute-co}, so for the sake of conciseness we do not repeat all arguments in full detail, but rather concentrate on the passages where the inclusion of covariates creates significant difficulties. The uniform Potter and Segers bounds from the previous sections, together with some other subtle considerations, are the main tools to help overcome these additional complications.

We first adopt the uniform big-o and small-o notation introduced in the latter reference, with a slight adaptation to include covariates. 
\begin{definition}[Uniform o-notation]
For the random vectors 
\begin{align*}
\bar{Z}^\ast_k&=(Z^\ast_{1,k},\dots,Z^\ast_{k,k},X^\ast_{[1,k]},\dots, X^\ast_{[k,k]})\\
&=( Z_{n-k+1,n}/Z_{n-k,n},\ldots , Z_{n,n}/Z_{n-k,n},X_{[n-k+1,n]},\dots, X_{[n,n]}),
\end{align*}
an array of functions $(f_{k,n})_{k,n\in\amsmathbb{N}}$, the baseline random variable $Z_{n-k,n}$, and a deterministic sequence $a_1,a_2,\ldots$ we say that $f_{k,n}(\bar{Z}^\ast_k)$  is $o_{{P}}(a_k)$ uniformly in large $t$ w.r.t. $Z_{n-k,n}$ if for every $\varepsilon>0$,
\begin{align}
\lim_{n/k,\,k\to\infty} 
\limsup\limits_{t\rightarrow\infty}\pr(|f_{k,n}(\bar{Z}^\ast_k)/a_k|>\varepsilon|\,Z_{n-k,n}=t) = 0.
\end{align} 
We then write $f_{k,n}(\bar{Z}^\ast_k)=\overline{o}_{P}(a_k)$. Analogously $f_{k,n}(\bar{Z}^\ast_k)=\overline{O}_{P}(a_k)$ using the notions of bounded in probability instead of convergence to zero in probability.
\end{definition}

\subsection{Proof of Theorem \ref{3.2} (main decomposition)}
We begin by imposing an additional restriction, which is thereafter lifted to prove the general decomposition.
\begin{condition}[Truncated version] There exists a function $\hat{\varphi}(w)$ such that $|\varphi(x,w)|\leq \hat{\varphi}(w)$ and $\hat{\varphi}(w)=0$ for $w>T$ and some $T<\infty$.
    \label{temp}
\end{condition}

\begin{theorem}[Main decomposition, truncated version]
    Let $\varphi$ have a uniform envelope $\hat{\varphi}(w)$ and satisfy Condition \ref{temp}. Further assume that $U_{F_{Y|X=x}}$ can be uniformly decomposed. Then
    \begin{align}
       & r_{k,n}=\nonumber\\
        &S_{k,n}(\phi)-\Big\{ \frac{1}{k}\sum_{i=1}^k \varphi(X^*_{i},V_i^*)\gamma_0^t(V_i^*)\delta_i^*+\frac{1}{k}\sum_{i=1}^k\gamma_1^t(V_i^*)(1-\delta_i^*)-\frac{1}{k}\sum_{i=1}^k\gamma_2^t(V_i^*)\Big\}
        \label{thm1}
    \end{align}
    is $\overline{o}_\pr(k^{-1/2})$.
    \label{thm:cond}
\end{theorem}

\begin{proof}[Proof of Theorem \ref{thm:cond}]
    
Let $Z_{n-k,n}=t$. Note that $(V_i^*, \delta_i^*,X^*_i)_{i=1}^k$ is an iid sequence and we have that
\begin{align}
    \amsmathbb{F}^{t}_{k,n}(x,y)=\sum_{i=1}^kW_{ik}1\{X_{[i,k]}^*\leq x, V_{i,k}^*\leq y\}
\end{align}
where 
\begin{align*}
    W_{ik}=\frac{\delta^*_{[i:k]}}{n-i+1}\prod_{j=1}^{i-1}\Big[\frac{k-j}{k-j+1}\Big]^{\delta^*_{[j:k]}}
\end{align*}
This implies that we can follow the proof in \cite{stute-co}, where $Z_{i,n}$ is changed with $V_{i,k}^*$ and $\delta_{[i,n]}^*$ with $\delta_{[1,k]}^*$. From \cite{stute-co} we have
\begin{align*}
    &S_{k,n}^{\varphi}=\int \varphi d\amsmathbb{F}_{k,n}^{t}\\
    &=\int \varphi(x,w) \exp\Bigg\{k\int_0^{w}\log\left(1+\frac{1}{k(1-H_{k}^t(z))}\right)\Bigg\}H_{k}^{0,t}(\dd z)H_{k}^{11,t}(\dd x,\dd w)
\end{align*}
where 
\begin{align*}
    &H_{k}^{0,t}(z)=\frac{1}{k}\sum_{i=1}^k I(V^*_{i,k}\leq z, \delta_{[i,k]}^*=0) \\
    &H_{k}^{11,t}(x,z)=\frac{1}{k}\sum_{i=1}^k I(X^*_{[i,k]}\leq x, V^*_{i,k}\leq z, \delta_{[i,k]}^*=1) \\
    &H_{k}^{t}(z)=\frac{1}{k}\sum_{i=1}^k I(V^*_{i,k}\leq z).
\end{align*}
Since the functions above are stepwise constant, we rewrite
\begin{align*}
    S_{k,n}^{\varphi}&=\frac{1}{k}\sum_{i=1}^k \varphi(X_i^*, V_i^*)\gamma_0^t(V_i^*)\delta_i^*(1+B_{ik}^t+C_{ik}^t)\\
    &\quad+\frac{1}{k}\sum_{i=1}^k \frac{1}{2}\varphi(X_i^*, V_i^*)\delta_i^* e^{\Delta_i}(B_{ik}^t+C_{ik}^t)
\end{align*}
where 
\begin{align}
    B_{ik}^t&=k\int_1^{V_i^*}\log\left(1+\frac{1}{k(1-H_{k}^t(z))}\right)H_{k}^{0,t}(\dd z)-\int_1^{V_i^*}\frac{1}{1-H_{k}^t(z)}H^{0,t}_k(\dd z)\label{eq:BB}\\
    C_{ik}^t&=\int_1^{V_i^*}\frac{1}{1-H_{k}^t(z)}H^{0,t}_k(\dd z)-\int_1^{V_i^*}\frac{1}{1-H_{k}^t(z)}H^{0,t}(\dd z)\label{eq:CC}
\end{align}
and $\Delta_i$ is between the first term in the rhs of \eqref{eq:BB} and second term in the rhs of \eqref{eq:CC}.

Since $C_{ik}^t$ is the same as in the non-covariate case, it follows that
\begin{align*}
    &\frac{1}{k}\sum_{i=1}^k \varphi(X_{i}^*, V_i^*)\gamma_0^t(V_i^*)\delta_i^*C_{ik}^t\\
    &= -\int\int\int \frac{1_{v<u,v,w}\varphi(w)\gamma_0^t(w)}{(1-H^t(v))^2}H_k^t(du)H_k^{0,t}(\dd v)H_k^{11,t}(\dd x,\dd w) \\
    &\quad+2\int\int 1_{v<w}\frac{\varphi(x,w)\gamma_0^t(w)}{1-H^t(v)}H_k^{0,t}(\dd v)H_k^{11,t}(\dd x,\dd w) \\
    &\quad- \int\int 1_{v<w}\frac{\varphi(x,w)\gamma_0^t(w)}{1-H^t(v)}H^{0,t}(\dd v)H_k^{11,t}(\dd x, \dd w)+ r^1_{k,n}
\end{align*}
where
\begin{align*}
    &r^1_{k,n}=\\
    &\int\int \varphi(x,w)\gamma_0^t(w)I(z<w) \frac{(H_{k}^t(z)-H^t(z))^2}{(1-H^t(z))^2(1-H_k^t(z))}H_k^{0,t}(\dd z) H_k^{11,t}(\dd x, \dd w).
\end{align*}
\begin{lemma}
The term    \begin{align*}
        \int\int 1_{v<w}\frac{\varphi(x,w)\gamma_0^t(w)}{1-H^t(v)}(H_k^{0,t}(\dd v)-H^{0,t}(\dd v))(H_k^{11,t}(\dd x,\dd w)-H^{11,t}(\dd x,\dd w))
    \end{align*}
    is $\overline{o}_\pr(k^{-1/2})$.
\end{lemma}
\begin{proof}
    Let
    \begin{align*}
        V_{k,n}:&=\int \int 1_{v<w}\varphi(x,w)\gamma_0^t(w)/(1-H^t(v))H_k^{0,t}(\dd v)H_k^{11,t}(\dd x, \dd w) \\
        &=\frac{1}{k^2}\sum_{i=1}^k\sum_{j=1}^k 1_{V_i^*<V_j^*}(1-\delta_i^*)\delta_j^*\varphi(X_j, V_j^*)\gamma_0^t(V_j^*)/(1-H^t(V_i^*)).
    \end{align*}
    By following the proof of lemma A.2 in the \cite{EKM}, we have
    \begin{align*}
        U_{k,n}:=\frac{k}{k-1}V_{k,n}=\frac{2}{k(k-1)}\sum_c h\left( \begin{pmatrix}
            X_i^* \\ V^*_i \\ \delta_i^*
        \end{pmatrix}, \begin{pmatrix}
            X_j^* \\ V^*_j \\ \delta_j^*
        \end{pmatrix} \right)
    \end{align*}
    where $h: \amsmathbb{R}^{m+1+1}\times \amsmathbb{R}^{m+1+1} \rightarrow \amsmathbb{R}$ is given by
    \begin{align*}
        &h(\mathbf{x}, \mathbf{y})=\\
        &\frac{1}{2}\Big(\frac{I(x_2<y_2)(1-x_3)y_3\varphi(\mathbf{y_1},y_2)\gamma_0^t(y_2)}{1-H^t(x_2)} +\frac{I(y_2<x_2)(1-y_3)x_3\varphi(\mathbf{x_1},x_2)\gamma_0^t(x_2)}{1-H^t(y_2)}\Big)
    \end{align*}
    where $\mathbf{x_1}, \mathbf{y_1}\in \amsmathbb{R}^m$, and $\sum_c$ denotes summation over the $\binom{k}{2}$ combinations of 2 distinct elements $\{i_1,i_2\}$ from $\{1,\ldots,k\}$. Since $h(\mathbf{x},\mathbf{y})$ is symmetric, we have $U_{k,n}$ is a U-statistic. By the argumentation on page 23 in \cite{EKM} the results then follow if we show that the following quantities are asymptotically independent of $t$:
    \begin{align*}
        \theta^t=\E h\left( \begin{pmatrix}
            X_1^* \\ V^*_1 \\ \delta_1^*
        \end{pmatrix}, \begin{pmatrix}
            X_2^* \\ V^*_2 \\ \delta_2^*
        \end{pmatrix} \right) \\
        \xi_1 =\text{Var} h_1\left( \begin{pmatrix}
            X_1^* \\ V^*_1 \\ \delta_1^*
        \end{pmatrix} \right) \\
        \xi_2 = \text{Var} h\left( \begin{pmatrix}
            X_1^* \\ V^*_1 \\ \delta_1^*
        \end{pmatrix}, \begin{pmatrix}
            X_2^* \\ V^*_2 \\ \delta_2^*
        \end{pmatrix} \right)
    \end{align*}

    where $h_1(\mathbf{x})=\E h\left(\mathbf{x}, \begin{pmatrix}
            X_1^* \\ V^*_1 \\ \delta_1^*
        \end{pmatrix}\right)$.
    First note that we have $H^{11,t}(\dd x,\dd v)= H^{1,t}(\dd v) F_{X|Y=v}^t(\dd x)$, where $H^{1,t}(z):=\pr(Z<tz, \delta=1 | Z>t)$. So
    \begin{align}
        \theta^t&=\int\int_1^{\infty} \frac{1_{v<w}\varphi(x,w)\gamma_0^t(w)}{1-H^t(v)}H^{0,t}(\dd v)H^{11,t}(\dd x,\dd w)\nonumber\\\nonumber
        &= \int_1^{\infty}\int \int_1^w \varphi(x,w)\frac{1}{1-G^t(w)}\frac{1}{(1-G^t(v))}G^t(\dd v) F^t_{X|Y=w}(\dd x) H^{1,t}(\dd w)\\\nonumber
        &= \int_1^{\infty}\int \int_1^w \varphi(x,w)\frac{1}{(1-G^t(v))}G^t(\dd v) F^t_{X|Y=w}(\dd x) F^{t}_Y(\dd w)\\\nonumber
        &= \int_1^{\infty}\int -\varphi(x,w)\log(1-G^t(w)) F^t_{X|Y=w}(\dd x) F_Y^t(\dd w)\\\nonumber
        &=\int\int_1^{\infty} -\varphi(x,w)\log(1-G^t(w)) F^t_{Y|X=x}(\dd w) F^t_{X}(\dd x)\\
        &\rightarrow  \int\int_1^{\infty} \varphi(x,w)\frac{1}{\gamma_G\gamma_F(x)}\log(w)w^{-1/\gamma_F(x)-1}\dd w F_X^{\circ}(\dd x):=\theta.
        \label{eq:theta}
    \end{align}
    Hence the limit is asymptotically independent of $t$. Here, we have used the dominated convergence theorem. We show now that this was possible since we may bound the integrand. Fist consider $-\varphi(x,w)\log(1-G^t(w))$. By Potter bounds we have for large enough $T_1$, there exists $\epsilon_a,\epsilon_c>0$ such that
    \begin{align*}
        1-G^t(w)=\frac{1-G(tw)}{1-G(w)}\leq  (1+\epsilon_a)w^{-1/\gamma_G+\epsilon_c}
    \end{align*}
    for $t>T_1$. So for $t>T_1$, we have
    \begin{align*}
        &-\varphi(x,w)\log(1-G^t(w)) \le \varphi(x,w)(\log(w)(1/\gamma_G-\epsilon_c)-\log(1+\epsilon_a)) \\
        & \leq \varphi(x,w)(\log(w)(1/\gamma_G^L+\epsilon_c)-\log(1-\epsilon_a)) \\
        & \leq A_0I(w\in[1,1+\epsilon_c])+A_1I(w>1+\epsilon_c)\varphi(x,w)\log(w)
    \end{align*}
    for some positive constants $A_0,A_1$. Thus
    \begin{align*}
        &\int\int_1^{\infty} -\varphi(x,w)\log(1-G^t(w)) F^t_{Y|X=x}(\dd w) F^t_{X}(\dd x)\\
        &\leq A_2+A_3\int\int_1^{\infty} \varphi(x,w)\log(w) F^t_{Y|X=x}(\dd w) F^t_{X}(\dd x)
    \end{align*}
    for some positive constants $A_2,A_3$.    
    By Condition \ref{temp} there exists a $T>0$ such that $\varphi(x,w)=0$ for all $x$ and $w>T$. Since $\hat{\varphi}(w)$ exists, we can find an additional continuous function, $C_1(w)$, such that $C_1(w)\geq |\varphi(x,w)|$ and $C_1(w)=0$ for $w>T+\delta_1$ for some $\delta_1>0$. Now let $U_{F^t_{Y|X=X^t}}$ be the tail-quantile function of $F^t_{Y|X=x}$ at $x=X^t$, that is, a random variable. Note that the following identity holds:
    \begin{align*}
        &\int\int_1^{\infty} \varphi(x,w)\log(w) F^t_{Y|X=x}(\dd w|x)F^t_{X}(\dd x)\\
        &\leq\E\left[\E\left[C_1(U_{F^t_{Y|X=X^t}}(\xi))\log(U_{F^t_{Y|X=X^t}}(\xi)))|X^t\right]\right]
    \end{align*}
   where $\xi$ is a standard Pareto random variable, independent of $X^t$, and
    \begin{align*}
        U_{F^t_{Y|X=X^t}}(y)=\frac{U_{F^t_{Y|X=X^t}}(y/(1-{F^t_{Y|X=X^t}}(t)))}{U_{F^t_{Y|X=X^t}}(1/(1-{F^t_{Y|X=X^t}}(t)))}.
    \end{align*}
    By uniform Potter bounds, we have that for $\epsilon_{a_1},\epsilon_{c_1}>0$, the following holds for $t$ large enough, 
    \begin{align*}
        &\E\left[\E\left[C_1(U_{F^t_{Y|X=X^t}}(\xi))\log(U_{F^t_{Y|X=X^t}}(\xi)))\right]\right]\\
        &\leq \E\left[\E\left[C_1^*\log\left((1+\epsilon_{c_1})T^{\gamma^U+\epsilon_{a_1}}\right)|X^t\right]\right]<\infty
    \end{align*}
    where  $C_1^*=\max_{w\in [1,T+\delta_1]}C_1(w)$ which exists and is finite since $C_1$ is continuous. 
    Hence the dominated convergence theorem can be applied.

Similar can bounds for $\xi_1$ and $\xi_2$ be found, and we only show the latter. We have 
    \begin{align*}
        &\xi_2=\\
        &\int \int_1^{\infty} \frac{I(v<w)\varphi(x,w)^2\gamma_0^t(w)^2}{(1-H^t(v))^2}H^{0,t}(\dd v)H^{11,t}(\dd w)-(\theta^t)^2 \\
        &= \int \int_1^{\infty}\int_1^w \varphi(x,w)^2 \frac{1-F^t(v)}{(1-H^t(v))(1-G^t(w))}G^t(\dd v)F_{Y|X=x}^t(\dd w)F^t_{X}(\dd x)-(\theta^t)^2 \\
        &\rightarrow  \int \frac{\gamma_H(x)}{\gamma_G(x)\gamma_F(x)} \int_1^{\infty} \varphi(x,w)^2 w^{1/\gamma_H(x)-1}\dd w F_X^{\circ}(\dd x)-\theta^2,
    \end{align*}
as desired. The result now follows as in \cite{EKM}.
\end{proof}

\begin{lemma}
The expression
    \begin{align*}
       & \int\int\int \frac{1_{v<u,v,w}\varphi(x,w)\gamma_0^t(w)}{(1-H^t(v))^2}\big(H_k^t(\dd u)H_k^{0,t}(\dd v)H_k^{11,t}(\dd x,\dd w)\\
        &\quad-H^t(\dd u)H^{0,t}(\dd v)H_k^{11,t}(\dd w)
        -H^t(du)H_k^{0,t}(\dd v)H^{11,t}(\dd x,\dd w)\\
        &\quad-H_k^t(\dd u)H^{0,t}(\dd v)H^{11,t}(\dd x,\dd w)+2H^t(\dd u)H^{0,t}(\dd v)H^{11,t}(\dd x,\dd w)\big)
    \end{align*}
    is $\overline{o}_\pr(k^{-1/2})$.
\end{lemma}

\begin{proof}
    Similar to the previous lemma.
    \end{proof}

By the two previous lemmas we have that $\frac{1}{k}\sum_{i=1}^k \varphi(X_i^*, V_i^*)\gamma_0^t(V_i^*)\delta_i^*C_{ik}^t=\overline{O}_\pr(k^{-1/2})$. We turn to the remainder part.

\begin{lemma}
    \label{lemma:A3}
The remainder satisfies
    \begin{align*}
        |r_{k,n}^1|=\overline{O}_\pr(k^{-1}).
    \end{align*}
\end{lemma}

\begin{proof}
    We have by the temporary Condition \ref{temp}
    \begin{align*}
         |r_{k,n}^1|&\leq \int |\varphi(x,w)|\gamma_0^t(w) H_k^{11,t}(\dd x, \dd w)\\
        & \quad \times\int I(z<T)\frac{(H_{k}^t(z)-H^t(z))^2}{(1-H^t(z))^2(1-H_k^t(z))}H_k^{0,t}(\dd z) \\
         &=: I_{k,1}\times I_{k,2}
    \end{align*}
    It follows directly from \cite{EKM} that $I_{k,2}=\overline{O}_\pr(k^{-1})$. By the uniform Potter bounds and the temporary condition, we get
    \begin{align*}
        \pr(I_{k,1}>\epsilon)&\le \frac{1}{\epsilon} \int \int_1^{\infty} |\varphi(x,w)| F_{Y|X=x}^t(\dd w)F^t_{X}(\dd x) \\&\leq \frac{1}{\epsilon}K\int_1^{\infty} C_1(w) w^{-1/\gamma_F^U-1+\delta}\dd w 
    \end{align*}
    where $C_1(w)$ is some continuous function satisfying $C_1(w)\geq |\varphi(x,w)|$ and $C_1(w)=0$ for $w>T+\delta_1$ for some $\delta_1>0$. Hence, we have that $I_{k,1}=\overline{O}_\pr(1)$.
\end{proof}

\begin{lemma}The expression
    \begin{align*}
        S_{k,n}^1=\frac{1}{k}\sum_{i=1}^k \varphi(X_i^*, V_i^*)\gamma_0^t(V_i^*)\delta_i^*B_{ik}^t=\overline{O}_\pr(k^{-1}).
    \end{align*}
\end{lemma}
\begin{proof}
    Follows by Lemma A.6 in \cite{EKM}.
\end{proof}

\begin{lemma}
    The expression
    \begin{align*}
        S_{k,n}^2=\frac{1}{2k}\sum_{i=1}^k |\varphi(X_i^*, V_i^*)|\delta_i^*e^{\Delta_i}[B_{ik}^t+C_{ik}^t]^2=\overline{O}_\pr(k^{-1}).
    \end{align*}
\end{lemma}
\begin{proof}
    Follows by Lemma A.7 in \cite{EKM}.
\end{proof}
By collecting all the lemmas we complete the proof of the truncated case.
\end{proof}

The next lemma is a useful tool towards showing Theorem \ref{3.2}.
\begin{lemma}
    Assume that $U_{F_{Y|X=x}}$ can be uniformly decomposed and let Condition \ref{cond-full} hold. Then for large $t$, we have
    \begin{align}
        \int\varphi(x,v)^2 (\gamma^t_0(v))^2dH^{11,t}(x,v) <\infty\\
        \int_1^{\infty}  |\varphi|(x,v)\gamma_0^t(v)\sqrt{C^t} dH^{11,t}(\dd x,\dd v)<\infty
        \label{prop12}
    \end{align}
    where $$C^t(x)=\int_1^{x}\frac{G^t(\dd y)}{(1-H^t(y)(1-G^t(y))}.$$
    \label{Prop}
\end{lemma}

\begin{proof}
    By using the using similar arguments as to show \eqref{eq:theta}, we get
    \begin{align*}
        &\int \varphi(x,w)^2(\gamma_0^t(x))^2 H^{11,t}(\dd x,\dd w)\\
        &= \int_1^{\infty} \int \varphi(x,w)^2(\gamma_0^t(x))^2 F^t_{X|Y=w}(\dd x|w)H^{1,t}(\dd w)\\
        &=\int_1^{\infty}\int \varphi(x,w)^2 \frac{1}{(1-G^t(w))^2}F^t_{X|Y=w}(\dd x)H^{1,t}(\dd w) \\
        &=\int\int_1^{\infty} \varphi(x,w)^2 \frac{1}{(1-G^t(w))}F^t_{Y|X=x}(\dd w)F^t_{X}(\dd x)  \\
        &\leq  C_1\int_1^{\infty} \hat{\varphi}(w)^2 w^{1/\gamma_G^L-1/\gamma_F^U-1+\epsilon}\dd w < \infty
    \end{align*}
for some constant $C_1$. The last inequality follows from Condition \ref{cond-full}. For \eqref{prop12}, note that by using Potter bounds, we have
    \begin{align*}
         &\int_1^{\infty}  |\varphi(x,v)|\sqrt{C^t}\gamma_0(v) dH^{11,t}(\dd x,\dd w)\\
         &=\int |\varphi(x,v)|\sqrt{C^t}F_{X,Y}^t(\dd x,\dd w) \\
         &\leq \int_1^{\infty} \hat{\varphi}(w)w^{(2\gamma_G)^{-1}-(2\gamma_F^U)^{-1}-1+\epsilon/2}\dd w \\
         &\leq  \left(\int_1^{\infty} \hat{\varphi}(w)^2w^{1/\gamma_G-1/\gamma_F^U-2+\epsilon}\dd w\right)^{1/2}<\infty.
    \end{align*}
The first two inequalities follow from \cite{EKM}, while the last follows from Condition \ref{cond-full}.
\end{proof}

\begin{proof}[Proof of Theorem \ref{3.2}]
For a given $\epsilon>0$, pick large enough $T>0$ such that $\Tilde{\varphi}$ satisfies
\begin{align*}
    \int_1^{\infty} (\overline{\varphi-\Tilde{\varphi}})(w)^2w^{a(\epsilon)}\dd w\leq \epsilon.
\end{align*}
Such $T$ exists by Condition \ref{cond-full}. By Lemma \ref{Prop}, we can then find an $\epsilon_1>0$ vanishing together with $\epsilon$, such that
\begin{align}
    &\int  (\varphi-\Tilde{\varphi})(x,w)^2 (\gamma_0^t)^2 H^{11,t}(\dd x, \dd w)\leq \epsilon_1 \label{HH}
\end{align}
and
\begin{align}
\int |\varphi-\Tilde{\varphi}|(x,w) \sqrt{C^t} H^{11,t}(\dd x,\dd w)\le \epsilon_1 \label{abc}
\end{align}
for large $t$. A possible choice for $\Tilde{\varphi}(x,v)$ is $=1_{v\leq T}\varphi(x,v)$. By argumentation from \cite{EKM}, the theorem follows if we can show that for $\varphi_1=\varphi-\Tilde{\varphi}$, we have
\begin{align*}
    k^{1/2}\left( \int_1^{\infty}\varphi_1 \dd\amsmathbb{F}_{k,n}^t-\int_1^{\infty} \varphi_1 \dd F^t\right)=\overline{O}_\pr(\epsilon_1^{1/2}).
\end{align*}
 From \cite{stute-co} we have that 
\begin{align*}
    S_{k,n}^{\varphi}&=\int \varphi d\amsmathbb{F}_{k,n}^t\\
    &=\int \varphi(x,w) \exp\Bigg\{k\int_0^{w-}\log\left(1+\frac{1}{k(1-H_k^t(z)}\right)H_k^{0,t}(\dd z)\Bigg\}H_k^{11,t}(\dd x,\dd w) \\
    &= \sum_{i=1}^k \varphi(X_i^*,V_i^*)\delta_i^* \exp\Bigg\{k\int_0^{w-}\log\left(1+\frac{1}{k(1-H_k^t(z)}\right)H_k^{0,t}(\dd z)\Bigg\} \\
    &= \sum_{i=1}^k \varphi(X_i^*,V_i^*) \delta_i^*\gamma_0^t(V_i) \exp\big\{B_{ik}^t+C_{ik}^t\big\}
\end{align*}
so that
\begin{align}
     &k^{1/2}\left( \int_1^{\infty}\varphi_1 d\amsmathbb{F}_{k,n}^t-\int_1^{\infty} \varphi_1 dF^t\right)\\
     &= k^{-1/2}\sum_{i=1}^k \left(\varphi(X_i^*,V_i^*) \delta_i^*\gamma_0^t(V_i)- \int_1^{\infty}\varphi_1dF^t\right)\nonumber\\
     &\quad+ k^{-1/2}\sum_{i=1}^k \varphi(X_i^*,V_i^*) \delta_i^*\gamma_0^t(V_i) (\exp(B_{ik}^t+C_{ik}^t)-1). \label{eq:thm31}
\end{align}
Using \eqref{HH}, the variance of the first sum is seen to be
\begin{align*}
    &\mbox{Var}\left(k^{-1/2}\sum_{i=1}^k \left(\varphi_1(X_i^*,V_i^*) \delta_i^*\gamma_0^t(V_i)- \int_1^{\infty}\varphi_1dF^t\right)\right)\\
    &=\mbox{Var}(\varphi_1(X_i^*,V_i^*) \delta_i^*\gamma_0^t(V_i)) \\
    &\leq \E \left[(\varphi_1(X_i^*,V_i^*) \delta_i^*\gamma_0^t(V_i))^2\right] \leq \epsilon_1.
\end{align*}
Thus by Chebyshev's inequality the first term in \eqref{eq:thm31} is $\overline{O}_\pr(\epsilon_1^{1/2})$. 
Concerning the second term of \eqref{eq:thm31} first note that  
\begin{align*}
    &k^{-1/2}\sum_{i=1}^k \varphi(X_i^*,V_i^*) \delta_i^*\gamma_0^t(V_i) (\exp(B_{ik}^t+C_{ik}^t)-1)\\
    &\leq k^{-1/2}\sum_{i=1}^k |\varphi(X_{[i,k]},V^*_i)|\delta_i^*\gamma_0^t(V_i^*)(|B_{ik}^t|+|C_{ik}^t|)\exp(|B_{ik}^t|+|C_{ik}^t|).
\end{align*}
From \cite{EKM} we have that $\exp(|B_{ik}^t|+|C_{ik}^t|)=\overline{O}_\pr(1)$. So it remains to bound
\begin{align*}
    k^{-1/2}\sum_{i=1}^k |\varphi(X_{i}^*,V^*_i)|\delta_i^*\gamma_0^t(V_i^*)(|B_{ik}^t|+|C_{ik}^t|).
\end{align*}
Note that
\begin{align*}
    &k^{-1/2}\sum_{i=1}^k |\varphi(X^*_{i},V^*_i)|\delta_i^*\gamma_0^t(V_i^*)|B_{ik}^t|\\
    &\leq k^{-1}\sum_{i=1}^k |\varphi(X^*_{i},V^*_i)\delta_i^*\gamma_0^t(V_i^*)\left( \int_{-\infty}^{V^*_i}\frac{H_k^{0,t}}{(1-H_k^t(z))^2}\right)^{1/2}.
\end{align*}
But this expression is now bounded from above by \cite{stute} by using \eqref{abc} in place of (2.7). Similarly
\begin{align*}
    k^{-1/2}\sum_{i=1}^k |\varphi(X^*_{i},V^*_i)|\delta_i^*\gamma_0^t(V_i^*)|C_{ik}^t|
\end{align*}
is bounded from above by following the last part of the proof of Theorem 1 in \cite{stute}.
\end{proof}

\subsection{Proof of Theorem \ref{thm:weak-con} (weak consistency)}
\begin{proof}[Proof of Theorem \ref{thm:weak-con}]
By Theorem \ref{3.2}, we have
\begin{align*}
    S_{k,n}(\varphi)=\frac{1}{k}\sum_{j=1}^k W_j +r_{k,n}
\end{align*}
where $\{W_j\}_{j=1}^k$ are i.i.d given $Z_{n-k,n}=t$, and $\sqrt{k}r_{k,n}$ is $\overline{o}_\pr(1)$. Denote
\begin{align*}
    W^*=\varphi(X^*,V^*)\gamma_0^t(V^*)\delta^*+\gamma_1^t(V^*)(1-\delta^*)-\gamma_2^t(V^*)
\end{align*}
where $(X^*,V^*, \delta^*)\overset{d}{=}(X^*_1,V_1^*, \delta^*_1)$.
Note that we can bound $\E W^*$ by using the same arguments as in proof of Theorem \ref{thm:decomposition}, and hence by the dominated convergence theorem, we have   
\begin{align*}
    \E W^*&=\E [\varphi(X^*,V^*)\gamma^t_0 (V^*)\delta^*] \\
    &=\int \varphi(x,v) \frac{1}{1-G^t(v)}H^{11,t}(\dd x,\dd v) \\
    &= \int \int \varphi(x,v) \frac{1}{1-G^t(v)}F_{X|Y=v}^t(\dd x) H^{1,t}(\dd v)  \\
    &= \int \int \varphi(x,v) F^t_{X|Y=v}(\dd x)F_Y^t(\dd v) \\
    &= \int \int \varphi(x,v) F_{X,Y}^t(\dd x,\dd v) \\
    & \rightarrow \int \int \varphi(x,v) F_{X,Y}^{\circ}(\dd x,\dd v)
\end{align*}
where $F^{\circ}$ is the limit of $F^t$ as $t\rightarrow \infty$. Then by following \cite{EKM}, we have $S_{k,n}\stackrel{\pr}{\to} S_{\circ}(\varphi)$.
\end{proof}

\subsection{Proof of Theorem \ref{thm:normal1} (first CLT)}
\begin{proof}[Proof of Theorem \ref{thm:normal1}]
Follows directly from the analogous proof in \cite{EKM}.
\end{proof}

\subsection{Proof of Theorem \ref{THM:nlast} (second CLT)}
We begin with two preliminary lemmas. Here we let $L_{Y|X=x}(y)$ have the Karamata representation 
\begin{align*}
        h(x,y)=c(x,y)\exp\left( \int_1^x \delta(x,u)/u \,\dd u \right), \quad x\geq 1
    \end{align*}
    where $c$ (is positive) and $\delta$ are measurable function, where $c(x,y)\rightarrow c(x)$ uniformly when $y\rightarrow \infty$ and $\delta(x,y)\rightarrow \delta(x)$ uniformly when $y\rightarrow \infty$. 
\begin{lemma}
    Let $A(t)$ be a family of subsets of $\mathcal{X}$, and define $\gamma_F^A(t)=\max_{x\in A(t)} \gamma_F(x)$. Assume that $\delta(x,u)$ and $c(x)$ are bounded, and that there exists $T>0$ such that for all $u>T$ it holds that $-u^{-\epsilon}<\delta(x,u)<u^{-\epsilon}$ for some $\epsilon>0$, and . Then
    \begin{align*}
         \pr(X\in A(t)| Y>t)\rightarrow 0
    \end{align*}
    if $t^{-1/\gamma_F^A(t)+1/\gamma_F^U}\rightarrow 0$.
    \label{lemma:prop}
\end{lemma}

\begin{proof}
  Note that for all $\epsilon>0$, we have for  $t>T$ that
  
\begin{align*}
    \exp\left(\int_1^t \delta(x,u)/u\dd u\right)&=\exp\left(\int_1^T \delta(x,u)/u\dd u\right)\exp\left(\int_T^t \delta(x,u)/u\dd u\right)\\
    &\leq K \exp\left(\int_T^t 1/u^{1+\epsilon}\dd u\right) \\
    &\leq K\exp\left(\frac{1}{\epsilon T^{\epsilon}}\right).
\end{align*}
for some $K>0$ since $\delta$ is bounded.
Hence $ \exp\left(\int_1^t \delta(x,u)/u\dd u\right)$ is bounded from above uniformly in $x$ and independently of $t$. 
 Similarly, it can be shown that the expression is bounded from below. By using that $c(x)$ is bounded, we can concluded that $L_{Y|X=x}(t)$ is bounded. So we have, using the notation of Lemma \ref{FX-dist}, that
\begin{align*}
    \pr(X\in A(t)| Y>t)&\leq \frac{\pr(X\in A(t), Y>t)}{\pr(Y>t,X\in B)}\\
    & =\frac{\int_{A(t)} \pr(Y>t|X=x)F_X(\dd x)}{\int_{B} \pr(Y>t|X=x)F_X(\dd x)} \\
    &\leq \frac{\int_A L_{Y|X=x}(t) F_X(\dd x) t^{-1/\gamma_F^A(t)}}{\int_B L_{Y|X=x}(t)F_X(\dd x) t^{-1\gamma_F^U}} \\
    &\leq K_1\frac{t^{-1/\gamma_F^A(t)}}{ t^{-1\gamma_F^U}}
\end{align*}
 where $K_1$ is some positive constant. This goes to 0 if $t^{-1/\gamma_F^A+1/\gamma_F^U}\rightarrow 0$.
\end{proof}

\begin{lemma}
    Let $n(k)=a^{k^b}$ for some $a>1$ and $b>1/2$. Then for any $\epsilon_1>0$
    \begin{align}
    &\sqrt{k}g(k)\rightarrow 0 \label{c1}\\
    &\sqrt{k}\left(\frac{n}{k}\right)^{(\gamma_H+\epsilon_1)(-g(k))}\rightarrow 0 \label{c2}\\
    &\sqrt{k}\left(\frac{n}{k}\right)^{(\gamma_H+\epsilon_1)(-1/(\gamma^U-g(k))+1/\gamma^U} \rightarrow 0. \label{c3}
\end{align}
    \label{lemma:g}
\end{lemma}
\begin{proof}
\eqref{c1} holds if we let $g(k)=\frac{1}{k^{1/2+\epsilon_2}}$ for some $\epsilon_2>0$. Assume for the moment that $a$ and $b$ are arbitrary constant, and let $n(k)=a^{k^b}$. We examine the following limits:
\begin{align*}
    &\sqrt{k}\left(\frac{a^{k^b}}{k}\right)^{(\gamma_H+\epsilon_1)(-g(k))}\rightarrow 0 \\
    &\sqrt{k}\left(\frac{a^{k^b}}{k}\right)^{(\gamma_H+\epsilon_1)(-1/(\gamma^U-g(k))+1/\gamma^U)} \rightarrow 0,
\end{align*}
which hold whenever
\begin{align}
    &\sqrt{k}a^{-g(k)k^b}\rightarrow 0 \label{c4} \\
    &\sqrt{k}\left(a^{k^b}\right)^{(-1/(\gamma^U-g(k))+1/\gamma^U)} \rightarrow 0. \label{c5}
\end{align}
For $\epsilon_3>\epsilon_2$ we have that 
\begin{align*}
    -g(k)=\frac{-1}{k^{1/2+\epsilon_2}}<\frac{-1}{k^{1/2+\epsilon_3}}.
\end{align*}
Furthermore, note that for arbitrary $c,d>0$, we have
\begin{align*}
    \frac{\frac{1}{\gamma^U}-\frac{1}{\gamma^U-1/k^{c}}}{-1/k^d}&=\frac{k^d}{\gamma^U(\gamma^Uk^c-1)},
\end{align*}
which converges to infinity for $d>c$. In other words, $\frac{1}{\gamma^U}-\frac{1}{\gamma^U-1/k^{c}}$ goes slower to $0$ than $-1/k^d$. Since both of these quantities are negative, we must have for $k$ large enough that
\begin{align*}
    \frac{1}{\gamma^U}-\frac{1}{\gamma^U-1/k^{c}}<-1/k^d
\end{align*}
for $d>c$. So
\begin{align*}
     \frac{1}{\gamma^U}-\frac{1}{\gamma^U-1/k^{0.5+\epsilon_2}}<-1/k^{1/2+\epsilon_3}.
\end{align*}
Consequently \eqref{c4} and \eqref{c5} hold if
\begin{align*}
    \sqrt{k}a^{k^b(-\frac{1}{k^{0.5+\epsilon_3}})}\rightarrow 0.
\end{align*}
This happens for $b>1/2+\epsilon_3$ and any $a>1$. Since all the $\epsilon$'s may be chosen arbitrarily small, the result holds. 
\end{proof}

\begin{proof}[Proof of Theorem \ref{THM:nlast}]
Note that given $X^t$ then ${U_{F_{Y|X=X^t}}(s\xi)}/{U_{F_{Y|X=X^t}}(s)}$ has cdf $F^t_{Y|X=X^t}$ when $s=U_{F_{Y|X=X^t}}^{\leftarrow}(t)$ and where $\xi$ is a standard Pareto random variable, independent of $X^t$. Hence we have
\begin{align*}
    \E \left[ \varphi\left(X^t, \frac{U_{F_{Y|X=X^t}}(s\xi)}{U_{F_{Y|X=X^t}}(s)}\right)\bigg| X^t\right]=\E \left[ \varphi\left(X^t, Y^t \right)\bigg| X^t\right],
\end{align*}
which implies that 
\begin{align*}
    \int\varphi \dd F^t=&\E\left[\E \left[ \varphi\left(X^t, \frac{U_{F_{Y|X=X^t}}(s\xi)}{U_{F_{Y|X=X^t}}(s)}\right)\bigg| X^t\right]\right]\\
    =&\E\left[ \varphi\left(X^t, \frac{U_{F_{Y|X=X^t}}(s\xi)}{U_{F_{Y|X=X^t}}(s)}\right)\right].
\end{align*}
Consequently
\begin{align*}
    \sqrt{k}\int \varphi d(F^{Z_{n-k,n}}-F^{U_H(n/k)})\stackrel{\pr}{\to}0
\end{align*}
may be written as 
\begin{align*}
    \sqrt{k}\Bigg(\E\Bigg[& \varphi\left(X^{Z_{n-k,n}}, \frac{U_{F_{Y|X=X^{Z_{n-k,n}}}}(y\xi)}{U_{F_{Y|X=X^{Z_{n-k,n}}}}(y)}\right)\Bigg| Z_{n-k,n}\Bigg]\\
    -&\E \left[ \varphi\left(X^{U_H(n/k)}, \frac{U_{F_{Y|X=X^{U_H(n/k)}}}(z\xi)}{U_{F_{Y|X=X^{U_H(n/k)}}}(z)}\right)\right]\Bigg)\stackrel{\pr}{\to}0
\end{align*}
where $y=U_{F_{Y|X=X^{Z_{n-k,n}}}}^{\leftarrow}(Z_{n-k,n})$ and $z=U_{F_{Y|X=X^{U_H(n/k)}}}^{\leftarrow}(U_H(n/k))$. For notation's sake let $a=Z_{n-k,n}$ and $b=U_H(n/k)$. With the use of the second part of Condition \ref{cond:assyp} and the mean value theorem, note that
\begin{align*}
    &\Bigg|\E \left[ \varphi\left(X^a, \frac{U_{F_{Y|X=X^a}}(y\xi)}{U_{F_{Y|X=X^a}}(y)}\right)- \varphi\left(X^b, \frac{U_{F_{Y|X=X^b}}(z\xi)}{U_{F_{Y|X=X^b}}(z)}\right)\Bigg| Z_{n-k,n}\right]\Bigg| \\
    &\leq M\Bigg|\E \left[ \hat{\varphi}\left(\frac{U_{F_{Y|X=X^a}}(y\xi)}{U_{F_{Y|X=X^a}}(y)}\right)- \hat{\varphi}\left(\frac{U_{F_{Y|X=X^b}}(z\xi)}{U_{F_{Y|X=X^b}}(z)}\right)\bigg| Z_{n-k,n}\right]\Bigg| \\
    &\leq M\Bigg|\E \left[ \hat{\varphi}'(y^*)\left(\frac{U_{F_{Y|X=X^a}}(y\xi)}{U_{F_{Y|X=X^a}}(y)}-\frac{U_{F_{Y|X=X^b}}(z\xi)}{U_{F_{Y|X=X^b}}(z)}\right)\bigg| Z_{n-k,n}\right]\Bigg|
\end{align*}
where $y^*$ is between $\frac{U_{F_{Y|X=X^a}}(y\xi)}{U_{F_{Y|X=X^a}}(y)}$ and $\frac{U_{F_{Y|X=X^b}}(z\xi)}{U_{F_{Y|X=X^b}}(z)}$. First note that 
 large enough $l$ and $\epsilon>0$, we have
\begin{align*}
    U_{F_{Y|X=x}}^{\leftarrow}(l)    &=\frac{1}{1-F_{Y|X=x}(l)}\\
    &=\frac{1}{L_{Y|X=x}(l)l^{-1/\gamma_F(x)}}\\
    &\geq l^{1/(\gamma_F^U+\epsilon)}
\end{align*}
So $y$ and $z$ goes uniformly to $\infty$. Hence we can conclude by Lemma \ref{segers} that 
\begin{align*}
    \hat{\varphi}'(y^*)\leq  C_1 \max\{\hat{\varphi}'(\xi^{\gamma_F^L-\epsilon}),\hat{\varphi}'(\xi^{\gamma_F^U+\epsilon})\}
\end{align*} with arbitrary high probability.
For notation's sake define
\begin{align*}
    &u(s,\xi,x):=U_{F_{Y|X=x}}(s\xi)/U_{F_{Y|X=x}}(s) \\
   &A(s,t,x^h):=\left(\int_s^{\infty} |\eta_{x^h}'(l)|\dd l+\int_t^{\infty} |\eta_{x^h}'(l)|\dd l\right)\\
   &B(t,x^h,x^z):=\int_t^{\infty} |\eta_{x^h}'(l)|+|\eta_{x^z}'(l)||\dd l.
\end{align*}
By Lemma \ref{segers}, we have
\begin{align*}
    &|u(y,\xi,x^h)-u(z,\xi,x^z)|\\
    &\leq2A(y,z,x^h) |\log(y/z)|\xi^{\gamma^U+\epsilon}+\xi^{\gamma^U+\epsilon}\log(\xi)\left(B(y,x^h,x^z)+|\eta_{x^z}-\eta_{x^h}|\right)
\end{align*}
for every $\epsilon>0$ and large $y$, $z$. Thus we get
\begin{align*}
    M&\sqrt{k}\E \Bigg[ \hat{\varphi}'(y^*)\Bigg(\frac{U_{F_{Y|X=X^a}}(y\xi)}{U_{F_{Y|X=X^a}}(y)}-\frac{U_{F_{Y|X=X^b}}(z\xi)}{U_{F_{Y|X=X^b}}(z)}\Bigg)\Bigg|Z_{n-k,n}\Bigg] \\
    &\leq \sqrt{k}MC_1\E\Big[\max\{\hat{\varphi}'(\xi^{\gamma_F^L-\epsilon}),\hat{\varphi}'(\xi^{\gamma_F^U+\epsilon})\}\xi^{\gamma^U+\epsilon}\Big]\\
    &\quad \times\Bigg\{\E\Bigg[2A\Bigg(U_{F_{Y|X=X^a}}^{\leftarrow}(Z_{n-k,n}),U_{F_{Y|X=X^b}}^{\leftarrow}(U_H(n/k)),X^a\Bigg)\\
    &\quad\quad\quad\times\Bigg|\log\Bigg(\frac{U_{F_{Y|X=X^a}}^{\leftarrow}(Z_{n-k,n})}{U_{F_{Y|X=X^b}}^{\leftarrow}(U_H(n/k))}\Bigg)\Bigg|\Bigg|Z_{n-k,n}\Bigg]\\
    &\quad\quad\quad+\E\left[\log(\xi)B\left(U_{F_{Y|X=X^b}}^{\leftarrow}(U_H(n/k)),X^a,X^b\right)\bigg| Z_{n-k,n}\right]\\
    &\quad\quad\quad+\E\big[\log(\xi)\big|\gamma(X^a)-\gamma(X^b)\big|\big|Z_{n-k,n}\big]\Bigg\}.
\end{align*}
with arbitrary high probability. By Condition \ref{cond:assyp} we have
\begin{align*}
    \E\Big[\max\{\hat{\varphi}'(\xi^{\gamma_F^L-\epsilon}),\hat{\varphi}'(\xi^{\gamma_F^U+\epsilon})\}\xi^{\gamma^U+\epsilon}\log(\xi)\Big]<\infty.
\end{align*}
So the proof is complete if 
\begin{align}
&2\sqrt{k}\E\Bigg[A\left(U_{F_{Y|X=X^a}}^{\leftarrow}(Z_{n-k,n}),U_{F_{Y|X=X^b}}^{\leftarrow}(U_H(n/k)),X^a\right)\label{pre}\\
&\quad\quad\quad\times\log\left(\frac{U_{F_{Y|X=X^a}}^{\leftarrow}(Z_{n-k,n})}{U_{F_{Y|X=X^b}}^{\leftarrow}(U_H(n/k)))}\right)\Bigg|Z_{n-k,n}\Bigg] \stackrel{\pr}{\to}0 \label{eq1}\\
&\sqrt{k}\E\big[\E\big[|\gamma(X^a)-\gamma(X^b)||Z_{n-k,n}\big]\stackrel{\pr}{\to}0 \label{eq2} \\
    &\sqrt{k}\E\left[B\left(U_{F_{Y|X=X^b}}^{\leftarrow}(U_H(n/k),X^a,X^b)\right)\bigg| Z_{n-k,n}\right] \stackrel{\pr}{\to} 0. \label{eq3}
\end{align}
Note that \eqref{eq3} follows from Condition \ref{cond:slowly} since 

\begin{align*}
    &\sqrt{k}\E\left[B\left(U_{F_{Y|X=X^b}}^{\leftarrow}(U_H(n/k),X^a,X^b)\right)\bigg| Z_{n-k,n}\right] \\
    &\quad\leq \sqrt{k}\sup_{x^h,x^z\in \mathcal{X}}B\left(U_{F_{Y|X=x^z}}^{\leftarrow}(U_H(n/k),x^h,x^z)\right)
\end{align*}
and that $\sqrt{k}/\log(n/k)\to 0$ by assumption. 
Now consider \eqref{eq2}. We have  
\begin{align*}
    &\sqrt{k}\E\left[|\gamma(X^a)-\gamma(X^b)|\bigg| Z_{n-k,n}\right]\\
    &\leq \sqrt{k}\E[|\gamma^U-\gamma(X^a)|]+\sqrt{k}\E[|\gamma^U-\gamma(X^b)|| Z_{n-k,n}].
\end{align*}
The first term satisfies
\begin{align*}
     &\sqrt{k}\E[\gamma^U-\gamma(X^a)]\\
     &=\sqrt{k}\E[(\gamma^U-\gamma(X^a))I(X^a \in A_{k})]+\sqrt{k}\E[(\gamma^U-\gamma(X^a))I(X^a \in A_{k}^c)]
\end{align*}
where $A_k:=\{x: |\gamma(x)-\gamma^U|>g(k)\}$. So we have
\begin{align*}
    \sqrt{k}\E[\gamma^U-\gamma(X^a)]\leq \sqrt{k}(\gamma^U-\gamma^L)P(X^a\in A_k)+ \sqrt{k}g(k).
\end{align*}
 We now consider $\sqrt{k}P(X^{U_H(n/k)}\in A_k)$, which converges to $0$ by Lemma \ref{lemma:prop}  if 
\begin{align*}
\sqrt{k}\frac{U_H(n/k)^{-1/(\gamma^U-g(k))}}{U_H(n/k)^{-1/\gamma^U}}
\end{align*}
converges to $0$. By Potter's    bounds we have
\begin{align*}
\sqrt{k}\frac{U_H(n/k)^{-1/(\gamma^U-g(k))}}{U_H(n/k)^{-1/\gamma^U}}&\leq\sqrt{k}M_1\left(\frac{n}{k}\right)^{(\gamma_H+\epsilon_1)(-1/(\gamma^U-g(k))+1/\gamma^U)}
\end{align*}
for some $\epsilon_1>0$, $M_1>0$ and $n/k$ large enough. So we require
\begin{align}
    &\sqrt{k}g(k)\rightarrow 0 \label{eq11}\\
    &\sqrt{k}\left(\frac{n}{k}\right)^{(\gamma_H+\epsilon_1)(-1/(\gamma^U-g(k))+1/\gamma^U)} \rightarrow 0
\end{align}
 and this is satisfied by Lemma \ref{lemma:g}. Now, we look at the second term $\sqrt{k}\E[\gamma^U-\gamma(X^b)|Z_{n-k,n}]$. Similarly, we get
 \begin{align*}
     \sqrt{k}\E[\gamma^U-\gamma(X^b)|Z_{n-k,n}]\leq \sqrt{k}(\gamma^U-\gamma^L)\E[P(X^a\in A_k)|Z_{n-k,n}]+ \sqrt{k}g(k).
 \end{align*}
 So, we need to check the first term. By using $U_H(\xi_{n-k,n})\stackrel{d}{=}Z_{n-k,n}$ for a sequence of independent standard Pareto variables $(\xi_i)$, we have with arbitrary high probability that have
 \begin{align*}
     \sqrt{k}\E[P(X^a\in A_k)|Z_{n-k,n}]&\leq \sqrt{k}M_2\left(\frac{n}{k}\right)^{(\gamma_H+\epsilon_1)(-1/(\gamma^U-g(k))+1/\gamma^U)} \\
     & \quad \times\left(\xi_{n-k,n}\left(\frac{n}{k}\right)^{-1}\right)^{(\gamma_H+\epsilon_1)(-1/(\gamma^U-g(k))+1/\gamma^U)}.
 \end{align*}
Since by Smirnov's lemma \cite{smirnov} $\xi_{n-k,n}\left(\frac{n}{k}\right)^{-1}\stackrel{\pr}{\to} 1$, we can conclude that \eqref{eq2} holds.

We now turn attention to \eqref{pre} and \eqref{eq1}. By same argumentation from \eqref{eq3}, we have that 

\begin{align}
    &2\sqrt{k}\E\Bigg[A\left(U_{F_{Y|X=X^a}}^{\leftarrow}(Z_{n-k,n}),U_{F_{Y|X=X^b}}^{\leftarrow}(U_H(n/k))\right)\nonumber\\
    &\quad\quad\quad\times\log\left(\frac{U_{F_{Y|X=X^a}}^{\leftarrow}(Z_{n-k,n})}{U_{F_{Y|X=X^b}}^{\leftarrow}(U_H(n/k)))}\right)\Bigg|Z_{n-k,n}\Bigg]\nonumber \\
    \leq &2\sqrt{k}\sup_{x^h,x^z\in \mathcal{X}}A\left(U_{F_{Y|X=x^h}}^{\leftarrow}(Z_{n-k,n}),U_{F_{Y|X=x^z}}^{\leftarrow}(U_H(n/k)),x^h\right)\label{eqAA}\\
    &\quad\quad\quad\times\E\Bigg[\log\left(\frac{U_{F_{Y|X=X^a}}^{\leftarrow}(Z_{n-k,n})}{U_{F_{Y|X=X^b}}^{\leftarrow}(U_H(n/k)))}\right)\Bigg|Z_{n-k,n}\Bigg]\label{eqUZ}.
\end{align}
\eqref{eqAA} goes to 0 in probability by condition \ref{cond:assyp}, so we just need to bound \eqref{eqUZ}. Note that
\begin{align}
   & \log\left(\frac{U_{F_{Y|X=X^a}}^{\leftarrow}(Z_{n-k,n})}{U_{F_{Y|X=X^b}}^{\leftarrow}(U_H(n/k)))}\right)\\
    &=\log\left(\frac{Z_{n-k,n}^{1/\gamma(X^a)}L_{X^b}(U_H(n/k))}{U_H(n/k)^{1/\gamma(X^a)}L_{X^a}(Z_{n-k,n})}\frac{U_H(n/k)^{1/\gamma(X^a)}}{U_H(n/k)^{1/\gamma(X^b)}}\right) \nonumber\\
    &=\log\left(\frac{Z_{n-k,n}^{1/\gamma(X^a)}L_{X^b}(U_H(n/k))}{U_H(n/k)^{1/\gamma(X^a)}L_{X^a}(Z_{n-k,n})}\frac{(L_H(n/k)(n/k)^{\gamma_H})^{1/\gamma(X^a)}}{(L_H(n/k)(n/k)^{\gamma_H})^{1/\gamma(X^b)}}\right)\nonumber \\
    &=\log\left(\frac{Z_{n-k,n}^{1/\gamma(X^a)}}{U_H(n/k)^{1/\gamma(X^a)}}\frac{L_{X^a}(U_H(n/k))}{L_{X^a}(Z_{n-k,n})}\right)\label{eq111}\\
    &\quad+\log(n/k)\left(\frac{1}{\gamma(X^a)}-\frac{1}{\gamma(X^b)}\right)\gamma_H+ \log\left(\frac{L_{X^b}(U_H(n/k))}{L_{X^a}(U_H(n/k))}\right)\label{eq222}\\
    &\quad+\log\left(\frac{L_H(n/k)^{1/\gamma(X^a)}}{L_H(n/k)^{1/\gamma(X^b)}}\right)\label{eq333},
\end{align}
where we used $U_{F_{Y|X=x}}^{\leftarrow}(s)=1/(1-F_{Y|X=x}(s))=1/(L_{F_{Y|X=x}}(s)s^{-\gamma_F(x)})$ and $U_{H}(s)=L_{H}(s)s^{\gamma_H}$, where $L_{H}$ is some slowly varying function. We now look at the four terms separately. By using $U_H(\xi_{n-k,n})\stackrel{d}{=}Z_{n-k,n}$ for a sequence of independent standard Pareto variables, $(\xi_i)$, we may re-express \eqref{eq111} as 
\begin{align*}
    &\frac{U_{F_{Y|X=X^a}}^{\leftarrow}(U_H(\xi_{n-k,n})))}{U_{F_{Y|X=X^a}}^{\leftarrow}(U_H(n/k))}\\
    &=\frac{L_{Y|X=X^a}(L_H(n/k)(n/k)^{\gamma_H})(L_H(n/k)(n/k)^{\gamma_H})^{-1/\gamma_F(X^a)}}{L_{Y|X=X^a}(L_H(\xi_{n-k,n})(\xi_{n-k,n})^{\gamma_H})(L_H(\xi_{n-k,n})(\xi_{n-k,n})^{\gamma_H})^{-1/\gamma_F(X^a)}}.
\end{align*}

Since $\xi_{n-k,n}\stackrel{\pr}{\to}\infty$, we use Uniform Potter's bounds to obtain, that for high enough $n/k$ then for every $\epsilon_a,\epsilon_c>0$ we have with arbitrary high probability that
\begin{align*}
    &\frac{L_{X=X^a}^F(L^{U^H}(n/k)(n/k)^{\gamma_H})(L^{U^H}(n/k)(n/k)^{\gamma_H})^{-1/\gamma_F(X^a)}}{L_{X=X^a}^F(L^{U^H}(\xi_{n-k,n})(\xi_{n-k,n})^{\gamma_H})(L^{U^H}(\xi_{n-k,n})(\xi_{n-k,n})^{\gamma_H})^{-1/\gamma_F(X^a)}}\\
    &\leq (1+\epsilon_c) \left(\frac{n/k}{\xi_{n-k,n}}\right)^{-\gamma_H/\gamma_F(X^a)+\epsilon_a} \\
    &\leq \sup_{x\in \mathcal{X}} (1+\epsilon_c) \left(\frac{n/k}{\xi_{n-k,n}}\right)^{-\gamma_H/\gamma_F(x)+\epsilon_a}.
\end{align*}

By Smirnov's Lemma \cite{smirnov}, we have $\frac{(n/k)}{\xi_{n-k,n}}\stackrel{\pr}{\to} 1$ and since $\gamma_F(x)\in [\gamma_F^L,\gamma_F^U]$, we have 

\begin{align*}
\E\left[\frac{U_{F_{Y|X=X^a}}^{\leftarrow}(Z_{n-k,n})}{U_{F_{Y|X=X^a}}^{\leftarrow}(U_H(n/k))}\Bigg| Z_{n-k,n}\right]
\end{align*}
 is eventually bounded in probability.

Now we look to the first term in \eqref{eq222}. If we include the multiplication of $A$ from \eqref{eqAA}, we need to show that
\begin{align}
&\sqrt{k}\sup_{x^h,x^z\in \mathcal{X}}A\left(U_{F_{Y|X=X^h}}^{\leftarrow}(Z_{n-k,n}),U_{F_{Y|X=X^z}}^{\leftarrow}(U_H(n/k)),x^h\right)\nonumber\\
    &\quad \times\log(n/k)\E\left[\left|\frac{1}{\gamma(X^a)}-\frac{1}{\gamma(X^b)}\right|\Bigg| Z_{n-k,n}\right]\stackrel{\pr}{\to} 0
    \label{eq:thm123}
\end{align}
By Condition \ref{cond:slowly}, the result follows if we can bound
\begin{align*}
    \sqrt{k}\E\left[\left|\frac{1}{\gamma(X^a)}-\frac{1}{\gamma(X^b)}\right|\Bigg| Z_{n-k,n}\right]
\end{align*}
in probability.
Using the same procedure as for \eqref{eq2}, we now need to bound two following equations
\begin{align*}
   &\sqrt{k}g(k)\\
    &\sqrt{k}(n/k)^{(\gamma_H+\epsilon_2)(-g(k))}
\end{align*}
for some $\epsilon_2>0$. These equations are satisfied by Lemma \ref{lemma:g}.
The second term in \eqref{eq222} follows from the convergence of the first term in \eqref{eq222} and Condition \ref{cond:slowly}. Finally \eqref{eq333} follows since $L_H(n/k)<n/k$ for large enough $n/k$ and hence holds if \eqref{eq:thm123} holds.
\end{proof}

\subsection{Proof of Theorem \ref{THM:last} (third CLT)}
\begin{proof}[Proof of Theorem \ref{THM:last}]
We may rewrite
\begin{align*}
    &\sqrt{k}\int \phi d(F^{Z_{n-k,n}}-F^{\circ})\\
        &=\sqrt{k}\E\Bigg[\varphi\Bigg(X^{Z_{n-k,n}},\frac{U_{F_{Y|X=X^{Z_{n-k,n}}}}(y\xi)}{U_{F_{Y|X=X^{Z_{n-k,n}}}}(y)}\Bigg) -\varphi(X^{\circ},\xi^{\gamma^U})\Bigg| Z_{n-k,n}\Bigg]\\
    &=\sqrt{k}\E \Bigg[\varphi\Bigg(X^{Z_{n-k,n}},\frac{U_{F_{Y|X=X^{Z_{n-k,n}}}}(y\xi)}{U_{F_{Y|X=X^{Z_{n-k,n}}}}(y)}\Bigg)\\
   &\quad\quad\quad\quad -\varphi\Bigg(X^{\circ},\frac{U_{F_{Y|X=X^{\circ}}}(y\xi)}{U_{F_{Y|X=X^{\circ}}}(y)}\Bigg)\Bigg|Z_{n-k,n}\Bigg]\\
    &\quad +\sqrt{k}\E \Bigg[\varphi\Bigg(X^{\circ},\frac{U_{F_{Y|X=X^{\circ}}}(y\xi)}{U_{F_{Y|X=X^{\circ}}}(y)}\Bigg) -\varphi(X^{\circ},\xi^{\gamma^U})\Bigg| Z_{n-k,n}\Bigg]\nonumber
\end{align*}
with $y=U_{F_{Y|X=X^{Z_{n-k,n}}}}^{\leftarrow}(Z_{n-k,n})$. 
By using lemma \ref{segers}, noticing that $\gamma_F(X^{\circ})=\gamma_F^U$ and using similar argumentation as in Theorem \ref{THM:nlast}, the first term converge to 0 in probability. Now we look at the latter term. By the mean value theorem, we have

\begin{align*}
    &\sqrt{k}\E\Bigg[\varphi\Bigg(X^{\circ},\frac{U_{F_{Y|X=X^{\circ}}}(y\xi)}{U_{F_{Y|X=X^{\circ}}}(y)}\Bigg)-\varphi(X^{\circ},\xi^{\gamma^U})\Bigg| Z_{n-k,n}\Bigg] \\
    & = \sqrt{k}\E\left[\frac{\partial}{\partial y}\varphi(X^{\circ},y^*(X^{\circ}))\left(\frac{U_{F_{Y|X=X^{\circ}}}(y\xi)}{U_{F_{Y|X=X^{\circ}}}(y)}-\xi^{\gamma^U}\right)\Bigg| Z_{n-k,n}\right] \\
    & \leq \sqrt{k}\E\left[\hat{\varphi}'(y^*(X^{\circ}))\left(\frac{U_{F_{Y|X=X^{\circ}}}(y\xi)}{U_{F_{Y|X=X^{\circ}}}(y)}-\xi^{\gamma^U}\right)\Bigg| Z_{n-k,n}\right]
\end{align*}
with $y^*(X^{\circ})$ between $\frac{U_{F_{Y|X=X^{\circ}}}(y\xi)}{U_{F_{Y|X=X^{\circ}}}(y)}$ and $\xi^{\gamma^U}$.
By uniform Potter's bounds
\begin{align*}
    \frac{1}{2}t^{\gamma_F^L-\epsilon}\leq \frac{U_{F_{Y|X=x}}(tl)}{U_{F_{Y|X=x}}(l)}\leq 2t^{\gamma_F^U+\epsilon}
\end{align*}
for large $l$ uniformly for all $x$. As argued in Theorem \ref{THM:nlast} $y$ goes uniformly to infinity, thus, we have
\begin{align}
    \hat{\varphi}'(y^*(X^{\circ}))\leq C_1 \max\{\hat{\varphi}'(0.5\xi^{\gamma_F^L-\epsilon}),\hat{\varphi}'(2\xi^{\gamma_F^U+\epsilon})\} \label{eq:hat}
\end{align}
with arbitrary large probability.

Now consider the term $\E\Big[\Big(\frac{U_{F_{Y|X=X^{\circ}}}(y\xi)}{U_{F_{Y|X=X^{\circ}}}(y)}-\xi^{\gamma^U_F}\Big)\Big|Z_{n-k,n}\Big]$. By Condition \ref{cond:second-order}, we can use Lemma \ref{lemma:segers2}  to get
\begin{align*}
    &\E\left[\left(\frac{U_{F_{Y|X=X^{\circ}}}(y\xi)}{U_{F_{Y|X=X^{\circ}}}(y)}-\xi^{\gamma^U_F}\right)\Bigg|Z_{n-k,n}\right]\leq \E\Big[K\xi^{\gamma_F^U}a_{X=X^{\circ}}(y)|Z_{n-k,n}\Big]
\end{align*}
for large $y$. So, collecting terms, we get that 
\begin{align}
    &\sqrt{k}M \left| \E\left[\varphi'(y^*)\left(\frac{U_{F_{Y|X=X^{\circ}}}(y\xi)}{U_{F_{Y|X=X^{\circ}}}(y)}-\xi^{\gamma^U_F}\right)\Bigg|Z_{n-k,n}\right]\right|\nonumber \\
    &\quad \leq \E\Big[\sqrt{k}Ka_{X=X^{\circ}}(U_{F_{Y|X=X^{\circ}}}^{\leftarrow}(Z_{n-k,n}))\label{eq:ff}\\
    &\quad\quad\quad\times\xi^{\gamma_F^U}\max\{\hat{\varphi}'(0.5\xi^{\gamma_F^L-\epsilon}),\hat{\varphi}'(2\xi^{\gamma_F^U+\epsilon})\}\Big| Z_{n-k,n}\Big] \label{eq:ss}
\end{align}
with arbitrary large probability. \eqref{eq:ss} is finite by Condition \ref{cond:assyp}. Since $a_{X=x}$ satisfies the uniform Potter's bounds property, we can use the bounds to obtain
\begin{align*}
    \frac{a_{X=x}(U_{F_{Y|X=x}}^{\leftarrow}(Z_{n-k,n}))}{a_{X=x}(U_{F_{Y|X=x}}^{\leftarrow}U_H((n/k))}=\frac{a_{X=x}(U_{F_{Y|X=x}}^{\leftarrow}U_H((\xi_{n-k,n}))}{a_{X=x}(U_{F_{Y|X=x}}^{\leftarrow}U_H((n/k))}\stackrel{\pr}{\to} 1 
\end{align*}
uniformly. So for \eqref{eq:ff} we have
\begin{align}
    &\E\Big[\sqrt{k}Ka_{X=X^{\circ}}(U_{F_{Y|X=X^{\circ}}}^{\leftarrow}(Z_{n-k,n}))\Bigg|Z_{n-k,n} \Bigg]\nonumber\\
    &\leq \sqrt{k}\left(\sup_{x\in \mathcal{X}} a_{X=x}(U_{F_{Y|X=x}}^{\leftarrow}(U_H(n/k)))\right)\sup_{x\in \mathcal{X}} \frac{a_{X=x}(U_{F_{Y|X=x}}^{\leftarrow}(Z_{n-k,n}))}{a_{X=x}(U_{F_{Y|X=x}}^{\leftarrow}(U_H(n/k)))}  \label{eq:aa}
\end{align}

This concludes for $\lambda(x)=0$, since $\sqrt{k}a_{X=x}(U_{F_{Y|X=x}}^{\leftarrow}(U_H(n/k)))$ goes uniformly to $\lambda(x)$.

Now for $\lambda(x)>0$, we have that
\begin{align*}
    \frac{\partial}{\partial y}\varphi(x,y^*)\stackrel{a.s.}{\to} \frac{\partial}{\partial y}\varphi(x,\xi^{\gamma_F}), \quad \text{as } y\rightarrow \infty, \quad x\in B,\\\quad \frac{\frac{U_{F_{Y|X=x}}(y\xi)}{U_{F_{Y|X=x}}(y)}-\xi^{\gamma^U}}{a_{X=x}(y)}\stackrel{a.s.}{\to}\xi^{\gamma_F^U}h_{\rho(x)}(\xi), \quad \text{as } y\rightarrow \infty, \quad x\in B. 
\end{align*}
The first convergence follows from the fact that $\frac{\partial}{\partial y}\varphi(x,y)$ is continuous and the second follows from Condition \ref{cond:second-order}. Let

\begin{align*}
   j(l):=\frac{\partial}{\partial y}\varphi(X^{\circ},l^*(X^{\circ}))\left(\frac{\frac{U_{F_{Y|X=x}}(l\xi)}{U_{F_{Y|X=x}}(l)}-\xi^{\gamma^U}}{a_{X=X^{\circ}}(l)}\right)
\end{align*}
where $l^*$ is between $\frac{U_{F_{Y|X=X^{\circ}}}(l\xi)}{U_{F_{Y|X=X^{\circ}}}(l)}$ and $\xi^{\gamma^U}$. By the argumentation up to \eqref{eq:ff} and \eqref{eq:ss}, we have that $j(l)$ is eventually bounded by a random variable with finite expectation. In addition $a_{X=x}(c_n(x))\sqrt{k}$ is eventually uniformly bounded by a constant, where $c_n(x)=(1-F_{Y|X=x}(U_H(n/k)))^{-1}$. Therefore the Dominated Convergence Theorem can be used to obtain
\begin{align*}
\E\left[a_{X=X^{\circ}}(c_n(X^{\circ}))\sqrt{k}j(l)\right]\to \E\left[ \frac{\partial}{\partial y}\varphi(X^{\circ},\xi^{\gamma_F^U})\xi^{\gamma_F ^U}h_{\rho(X^{\circ})}(\xi)\lambda(X^{\circ})\right].
\end{align*}
Now, we have
\begin{align*}
&\sqrt{k}\int \int\phi(x,v) (F^{Z_{n-k,n}}-F^{\circ})(\dd x,\dd v)\\
    &=\E\left[\sqrt{k}a_{X=X^{\circ}}(U_{F_{Y|X=X^{\circ}}}^{\leftarrow}(Z_{n-k,n})))j(U_{F_{Y|X=X^{\circ}}}^{\leftarrow}(Z_{n-k,n}))\Big| Z_{n-k,n}\right] \\
    &=\E\left[\sqrt{k}a_{X=X^{\circ}}(c_n(X^{\circ}))\frac{a(U_{F_{Y|X=X^{\circ}}}^{\leftarrow}(Z_{n-k,n})))}{a_{X=X^{\circ}}(c_n(X^{\circ}))}j(U_{F_{Y|X=X^{\circ}}}^{\leftarrow}(Z_{n-k,n}))\Bigg| Z_{n-k,n}\right] \\
    &= \sup_{x\in\mathcal{X}}\frac{a_{X=x}(U_{F_{Y|X=x}}^{\leftarrow}(Z_{n-k,n}))}{a_{X=x}(c_n(x))}\E\left[\sqrt{k}a_{X=X^{\circ}}(c_n(X^{\circ}))j(U_{F_{Y|X=x}}^{\leftarrow}(Z_{n-k,n}))\Bigg| Z_{n-k,n}\right] \\
    &\xrightarrow{\pr} \int \lambda(x)C(\gamma_F(x),\rho(x)) F^{\circ}_X(\dd x).
\end{align*}
\end{proof}

\newpage
\section{Proof of Lemma \ref{FX-dist} (asymptotic behaviour of the tail distribution of \textit{X})}
\label{sec:proof-X}

\begin{proof}[Proof of Lemma \ref{FX-dist}]

    Assume for simplicity that $\pr(X\in B)>0$, though the proof extends to the general case with the appropriate notation. Let us first assume that there exists a non-empty set $B_1\subset\mathcal{X}\setminus B$ with $\pr(X\in B_1)>0$. Note that for every $x\in B_1$ there exists a closed nonempty set $A\subset \mathcal{X}\setminus B$ with $P(X\in A)>0$. Hence, there exists a $\epsilon>0$ such that $-1/\gamma_F^A+\epsilon<-1/\gamma_F^U-\epsilon$, where $\gamma_F^A=\max_{x\in A} \gamma_F(x)$. Let the Karamata representation for $L_{Y|X=x}$ be    
    
    \begin{align*}
        c(x,y)\exp\left(\int_1^t \frac{\delta(x,u)}{u}\dd u\right).
    \end{align*}

    Note that for all $\epsilon>0$, we have for large enough $s$ and $t$ that
    \begin{align*}
        \exp\left(\int_1^t \frac{\delta(x,u)}{u}\dd u\right)&=\exp\left(\int_1^s \frac{\delta(x,u)}{u}\dd u\right)\exp\left(\int_s^t \frac{\delta(x,u)}{u}\dd u\right) \\
        &\leq \exp\left(\int_1^s \frac{\delta(x,u)}{u}\dd u\right) \exp\left(\int_s^t \frac{\epsilon}{u}\dd u\right) \\
        &= K t^{\epsilon}.
    \end{align*}
    for some $K>0$, which doesn't depend on $x$. In a similar fashion we can find a lower bound. Thus for large enough $t$ we have
    \begin{align*}
        &\pr(X\in A| Y>t)\\
        &=\frac{\pr(X\in A, Y>t)}{\pr(Y>t)}\\
        &\leq \frac{\int_A \pr(Y>t|X=x)F_X(\dd x)}{\int_B \pr(Y>t|X=x)F_X(\dd x)} \\
        &= \frac{\int_A L_{Y|X=x}(t) t^{-1/\gamma_F(x)}F_X(\dd x)}{\int_B L_{Y|X=x}(t) t^{-1/\gamma_F(x)}F_X(\dd x)}\\
        &\leq K_1 \frac{t^{-1/\gamma_F(x)+\epsilon}}{t^{-1/\gamma_F^U-\epsilon}}\\
       & \rightarrow  0
    \end{align*}
    for some $K_1>0$.  Now assume that $B_1\subset B$. We then have
    \begin{align*}
        &\lim_{t\to \infty}\frac{\pr(X\in B_1, Y>t)}{\pr(Y>t)}\\
        &\quad=\lim_{t\to \infty}\frac{\int_{B_1}L_{Y|X=x}(t)t^{-1/\gamma^U_F}F_X(\dd x)}{\int_{\mathcal{X}\setminus B} L_{Y|X=x}(t)t^{-1/\gamma_F(x)}F_X(\dd x)+\int_B L_{Y|X=x}(t)t^{-1/\gamma^U_F}F_X(\dd x)} \\
        &\quad=\lim_{t\to \infty}  \frac{\int_{B_1}L_{Y|X=x}(t)F_X(\dd x)}{\int_B L_{Y|X=x}(t)F_X(\dd x)}
    \end{align*}

    if it exists.
\end{proof}
\newpage

\section{Assumptions verification for simulation study}
\label{sec:Burr}

In this section we check that the Burr model as specified in Section \ref{sec:application} satisfies all the conditions assumed in the various theorems. Note that Condition \ref{cond-full} and \ref{cond:assyp} solely depend on $\varphi$ and the limit distribution. Since they follow easily they are not mentioned in this section. 

We assume that $\kappa(x)\in [\kappa_L,\kappa_U]$ and $\tau(x)\in [\tau_L,\tau_U]$, where $0<\kappa_L<\kappa_U<\infty$ and $0<\tau_L<\tau_U<\infty$. The calculations are equivalent with open intervals instead. 

\begin{itemize}
    \item \textit{Uniform Karamata representation}.\\
    We have that 
    \begin{align*}
    \Bar{F}_{Y|X=x}(y)&=\exp\left(-\kappa(x)\log(1+y^{\tau(x)})\right) \\
    &= \exp(-\log(2)\kappa(x))\exp\left(-\kappa(x)\int_1^y\frac{u^{\tau(x)}\tau(x)}{u(1+u^{\tau(x)})}\dd u\right).
\end{align*}

We have $c(x,y)=\exp(-\log(2)\kappa(x))$ and $\delta(u,x)=\kappa(x)\tau(x)\frac{u^{\tau(x)}}{(1+u^{\tau(x)})}$. $c$ does not depend on $y$, and hence it clearly converges uniformly. Note that 

\begin{align*}
    |\delta(u,x)-\kappa(x)\tau(x)|\leq \left|\frac{u^{\tau_L}}{(1+u^{\tau_L})}\kappa_U\tau_U-\kappa_U\tau_U\right|\to 0
\end{align*}

as $u\to \infty$. Hence it converges uniformly.

\item \textit{Definition \ref{def:uni-U} for $U_{F_{Y|X}}$}. \\
We have that
\begin{align*}
    U_{F_{Y|X=x}}(t)&=(t^{1/\kappa(x)}-1)^{1/\tau(x)} \\
    &=\exp\left(\frac{1}{\tau(x)}\log(t^{1/\kappa(x)}-1)\right) \\
    &=\exp\left(\frac{1}{\tau(x)}(\log(t^{1/\kappa(x)}-1)-\log(t^{1/\kappa(x)}-1))\right)\\
    &\quad\times \exp\left(\frac{1}{\tau(x)}\log(T^{1/\kappa(x)}-1)\right) \\
    &=\exp\left(\frac{1}{\tau(x)}\int_T^t \frac{u^{1/\kappa(x)}}{u\kappa(x)(u^{1/\kappa(x)}-1)}\dd u\right)U_{F_{Y|X=x}}(T).
\end{align*}
Thus $\eta_x(u)=\frac{u^{1/\kappa(x)}}{\kappa(x)\tau(x)(u^{1/\kappa(x)}-1)}$. Note

\begin{align*}
    \left|\eta_x(u)-\frac{u^{1/\kappa(x)}}{\kappa(x)\tau(x)(u^{1/\kappa(x)}-1)}\right|\leq \left| \frac{1}{\kappa_L\tau_L}\frac{u^{1/\kappa_U}}{(u^{1/\kappa_U}-1)}-\frac{1}{\kappa_L\tau_L}\right|.
\end{align*}

Hence it converges uniformly.

\item \textit{First part of Condition \ref{cond:slowly}}. \\
We have that $\eta_x(u)=\frac{u^{1/\kappa(x)}}{\kappa(x)\tau(x)(u^{1/\kappa(x)}-1)}$. Note that $\gamma_F(x)=\lim_{u\rightarrow \infty} \eta_x(u)$ and that $\eta_x(u)$ is increasing in $u$. We have $U_H(x)\geq x^{\gamma_H-\epsilon}$ for large enough $x$, where $\epsilon\in (0,\gamma_H)$ (and likewise eventually we have $U^{\leftarrow}_{Y|X=x_0}(y)\geq y^{1/\gamma_F^L-\epsilon_1}$ uniformly for $x_0$). So

\begin{align*}
    &\log(n/k)|\gamma_{F}(x_1)-\eta_{x_1}(U^{\leftarrow}_{x_0}(n/k)^{\gamma_H-\epsilon})|\\
    &\leq \log(n/k)|\gamma_{F}(x_1)-\eta_{x_1}((n/k)^{\gamma^H/\gamma^U+\epsilon_2})| \\
    &=\log(n/k)\left|\gamma_{F}(x_1)-\frac{(n/k)^{a/\kappa(x_1)}}{\tau(x_1)\kappa(x_1)((n/k)^{a/\kappa(x_1)}-1)}\right| \\
    &=\frac{1}{\tau(x_1)\kappa(x_1)}\log(n/k)\left|1-\left(\frac{(n/k)^{a/\kappa(x_1)}}{((n/k)^{a/\kappa(x_1)}-1)}\right)\right| \\
    & \leq \frac{1}{\tau_L\kappa_L}\log(n/k)\left|1-\left(\frac{(n/k)^{a/\kappa_U}}{((n/k)^{a/\kappa_U}-1)}\right)\right|
\end{align*}
where $a=\gamma^H/\gamma^U+\epsilon_2>0$. Note that
\begin{align*}
    \lim_{y\rightarrow \infty}\log(y)\left(1-\frac{y^b}{y^b-1}\right)&=\lim_{y\rightarrow \infty}\frac{\frac{y^bb}{y^b-1}}{\frac{1}{log(y)^2y}} \\
    &=\lim_{y\rightarrow \infty} \frac{y^b\log(y)^2b}{y^b-1}= 0
\end{align*}
for $b>0$. Hence the condition holds.

\item \textit{Second part of  Condition \ref{cond:slowly}}. \\
We have
\begin{align*}
    \Bar{F}_{Y|X=x}(y)&=\exp\left(-\kappa(x)\log(1+y^{\tau(x)})\right)\\
    &=\exp\left(-\kappa(x)\log(y^{\tau(x)})\right)\exp\left(-\kappa(x)\log\left(\frac{1+y^{\tau(x)}}{y^{\tau(x)}}\right)\right) \\
    &= y^{-1/\gamma_F(x)}\left(\frac{1+y^{\tau(x)}}{y^{\tau(x)}}\right)^{\kappa(x)}.
\end{align*}
So $L_{Y|X=x}(y)=\left(\frac{1+y^{\tau(x)}}{y^{\tau(x)}}\right)^{\kappa(x)}$, and
\begin{align*}
    \frac{L_{Y|X=x^b}(y)}{L_{Y|X=x^a}(y)}=\frac{\left(\frac{1+y^{\tau(x^a)}}{y^{\tau(x^a)}}\right)^{\kappa(x^a)}}{\left(\frac{1+y^{\tau(x^b)}}{y^{\tau(x^b)}}\right)^{\kappa(x^b)}}\leq \frac{\left(\frac{1+y^{\tau_L}}{y^{\tau_L}}\right)^{\kappa_U}}{\left(\frac{1+y^{\tau_U}}{y^{\tau_U}}\right)^{\kappa_L}}\to 1
\end{align*}
for all $x^a,x^b$. Likewise, can a lower limit be established. So the convergence holds uniformly.

\item \textit{Existence of $-u^{-\epsilon}<\delta_{L_{Y|X}}(x,u)<u^{-\epsilon}$}:  \\
We have $L_{Y|X=x}(y)=\left(\frac{1+y^{\tau(x)}}{y^{\tau(x)}}\right)^{\kappa(x)}$, hence we can let $c_{L_{Y|X}}(x,y)=\left(\frac{1+y^{\tau(x)}}{y^{\tau(x)}}\right)^{\kappa(x)}$ and $\delta_{L_{Y|X}}(x,u)=0$. The condition is then fulfilled. 

\item \textit{Condition \ref{cond:second-order}, and $a_{X=x}$ can be uniformly decomposed}. \\
By using Theorem 2.3.4 in \cite{dehaan} we have $a_{X=x}(t)=\rho(x)\left(1-\left(\frac{t^{1/\kappa(x))}}{t^{1/\kappa(x)}-1}\right)^{1/\tau(x)}\right)$, where $\rho(x)=-1/\kappa(x)$. We can rewrite $a_{X=x}$ such that 
\begin{align*}
    a_{X=x}(t)=-\frac{1}{\kappa(x)}t^{-1/\kappa(x)}\frac{t^{1/\kappa(x)}}{t^{1/\kappa(x)}-1}\left( 1+\frac{\frac{1}{1-t^{-1/\kappa(x)}}-\left(\frac{1}{1-t^{-1/\kappa(x)}}\right)^{1/\tau(x)}}{1-\frac{1}{1-t^{-1/\kappa(x)}}}\right).
\end{align*}
With respect to the uniform decomposition, we have that 
\begin{align*}
c(x,y)&=-\frac{1}{\kappa(x)}\frac{y^{1/\kappa(x)}}{y^{1/\kappa(x)}-1}\left( 1+\frac{\frac{1}{1-y^{-1/\kappa(x)}}-\left(\frac{1}{1-y^{-1/\kappa(x)}}\right)^{1/\tau(x)}}{1-\frac{1}{1-y^{-1/\kappa(x)}}}\right)
\end{align*}
This converges uniformly by the same strategy as above. Thus, we  conclude that $a_{X=x}$ can be uniformly decomposed.

Note that 
\begin{align*}
 \frac{U_{F_{Y|X=x}}(ty)/U_{F_{Y|X=x}}(t)-y^{\gamma_F(x)}}{a_{X=x}(t)}=\frac{\frac{((ty)^{1/\kappa(x)}-1)^{-1/\tau(x)}}{(y^{1/\kappa(x)}-1)^{-1/\tau(x)}}-y^{\gamma_F(x)}}{\rho(x)\left(1-\left(\frac{t^{1/\kappa(x)}}{t^{1/\kappa(x)}-1}\right)^{1/\tau(x)}\right)}.
\end{align*}
which with a similar approach can be shown to converge uniformly. 

\end{itemize}

\end{document}